\documentclass[a4paper]{article}
\UseRawInputEncoding

\usepackage{amsthm}

\newtheorem{defi}{Definition}[section]
\newtheorem{thm}{Theorem}[section]
\newtheorem{lem}{Lemma}[section]

\newtheorem{cor}{Corollary}[section]

\usepackage[toc,page]{appendix}
\usepackage{url}
\usepackage{pgfplots}
\usepackage{tikz}
\usepackage{amsmath}
\usepackage{mathtools}
\usepackage[a4paper, left=3cm, right=3cm, top=4cm]{geometry}
\usepackage{amssymb}
\usepackage{amsfonts}
\usepackage{scrlayer-scrpage}
\usepackage{hyperref}
\usepackage{accents}
\usepackage{ulem}
\clearscrheadfoot
\newcommand{\R}{\mathbb{R}}

\newcommand{\C}{\mathbb{C}}

\def\Xint#1{\mathchoice
{\XXint\displaystyle\textstyle{#1}}
{\XXint\textstyle\scriptstyle{#1}}
{\XXint\scriptstyle\scriptscriptstyle{#1}}
{\XXint\scriptscriptstyle\scriptscriptstyle{#1}}
\!\int}
\def\XXint#1#2#3{{\setbox0=\hbox{$#1{#2#3}{\int}$ }
\vcenter{\hbox{$#2#3$ }}\kern-.6\wd0}}

\def\dashint{\Xint-}

\AtBeginDocument{
  \let\div\relax
  \DeclareMathOperator{\div}{div}
}

\usepackage{setspace}
\makeatletter
\newcommand{\MSonehalfspacing}{%
  \setstretch{1.44}%  default
  \ifcase \@ptsize \relax % 10pt
    \setstretch {1.448}%
  \or % 11pt
    \setstretch {1.399}%
  \or % 12pt
    \setstretch {1.433}%
  \fi
}
\newcommand{\MSdoublespacing}{%
  \setstretch {1.92}%  default
  \ifcase \@ptsize \relax % 10pt
    \setstretch {1.936}%
  \or % 11pt
    \setstretch {1.866}%
  \or % 12pt
    \setstretch {1.902}%
  \fi
}
\makeatother

\ohead{\pagemark}

\begin{document}

	\title{Distributional Fractional Gradients and a Bourgain-Brezis-type Estimate}
	
	\author{Jerome Wettstein\footnote{Department of Mathematical Sciences, Florida Institute of Technology, 32901 Melbourne, USA}}
\maketitle
\date{ }
\begin{abstract}
In this paper, we extend the definition of fractional gradients found in Mazowiecka-Schikorra \cite{schikorra} to tempered distributions on $\R^n$, introduce associated regularisation procedures and establish some first regularity results for distributional fractional gradients in $L^{1}_{od}$. The key feature is the introduction of a suitable space of off-diagonal Schwarz functions $\mathcal{S}_{od}(\R^{2n})$, allowing for a dual definition of the fractional gradient on an appropriate space of distributions $\mathcal{S}^{\prime}_{od}(\R^{2n})$ by means of fractional divergences defined on $\mathcal{S}_{od}(\R^{2n})$. In the course of the paper, we make a first attempt to define Sobolev spaces with negative exponents in this framework and derive a result reminiscent of Bourgain-Brezis \cite{bbdivY} and Da Lio-Rivi\`ere-Wettstein \cite{wett} in the form of a fractional Bourgain-Brezis inequality for this kind of gradient.
\end{abstract}
\medskip

\tableofcontents

	\section{Introduction}
	\label{introduction}

	In the paper Bourgain-Brezis \cite{bbdivY}, Bourgain and Brezis investigate the solvability of divergence equations $\div Y = f$ for critical Sobolev exponents and deduce an improved regularity property for "special" solutions in a general context. In fact, they prove that boundedness may be assumed for at least one solution, improving regularity in a way that cannot be derived from the well-known Sobolev embeddings, which allow for a general $BMO$-estimate at best. The argument to do so in dimension $2$ involves a duality argument employing an estimate for gradients in $L^{1} + H^{-1}$. Similar inequalities were later obtained in Da Lio-Rivi\`ere-Wettstein \cite{wett} to establish inequalities which culminated in a boundary value characterisation of Bergman spaces reminiscent of the classical analogue for Hardy spaces. It should be noted that the operators studied in the paper Da Lio-Rivi\`ere-Wettstein \cite{wett} are closely related to fractional gradients as  introduced in Shieh-Spector \cite{shiehspec}, if one uses Riesz potentials in order to move the fractional Laplacians from the right to the left handside of the inequalities found in \cite{wett}. Meanwhile, in the paper Mazowiecka-Schikorra \cite{schikorra}, Mazowiecka and Schikorra introduced the intriguing and very natural notion of fractional gradient and fractional divergence. These are closely connected to the Gagliardo-Sobolev norms and distinct from the one in Shieh-Spector \cite{shiehspec}. The fractional gradient in Mazowiecka-Schikorra \cite{schikorra} is based on difference quotients and thus depends on more variables than the initial function, while the one in Shieh-Spector \cite{shiehspec} is based on certain integral operators and thus does not include dependence on more variables than the initial function. These ideas are then used in Da Lio-Mazowiecka-Schikorra \cite{gaugeriviere} and Mazowiecka-Schikorra \cite{schikorra} to study non-local Wente-type estimates for fractional div-curl quantities and extending Rivi\`ere's celebrated gauge construction from the regularity theory of harmonic maps to the realm of fractional harmonic maps, further emphasising the substantial similarities between these cases.\\
	
	In the current paper, we contribute to the understanding of Mazowiecka and Schikorra's fractional gradients and divergences by forumlating a distributional setup adapted to the study of fractional gradients in a general setting. We show that we may define the fractional gradient by means of duality, as the fractional divergence defines a continuous operator on the test function side. Moreover, we introduce regularisation techniques which bridge the gap between functions of $n$ variables and functions of $2n$ variables, a key step in establishing smooth approximations that are natural both for functions of $n$ as well as of $2n$ variables. Later, we exploit this setup to prove a first regularity result for $L^{1}_{od}$-fractional gradients and attempt to generalise the notion of Sobolev spaces with negative exponent to this situation, in order to obtain a Bourgain-Brezis-type inequality similar to those in Bourgain-Brezis \cite{bbdivY} and Da Lio-Rivi\`ere-Wettstein \cite{wett}. We succeed with a natural space $H^{-1/2}_{od}$, however, there are other definitions possible.\\
	
	The general outline of the paper is as follows: First, in Section \ref{prelim}, we recall some of the most important properties and notions associated with tempered distributions. In Section \ref{offdiagdist}, we introduce the previously mentioned class of off-diagonal Schwartz functions and study some of their key properties, such as approximations and connections to the fractional gradients in Mazowiecka-Schikorra \cite{schikorra} as well as the fractional Laplacians. Afterwards, in Section \ref{bbest}, we study a version of the Bourgain-Brezis estimates from Bourgain-Brezis \cite{bbdivY} in analogy to Da Lio-Rivi\`ere-Wettstein \cite{wett}. Thus, we also discuss natural ways to introduce $H^{-1/2}_{od}$ as a subset of the space of off-diagonal distributions. During the course of our endeavours, we shall extend the Sobolev embeddings to distributions with fractional gradients satisfying the natural integrability properties. Finally, we very briefly extend the remaining Sobolev embeddings in Section \ref{sobolevemb} to the framework of fractional gradients.	
	
	\section{Preliminaries}
	\label{prelim}
	
	Before we enter our discussion, let us briefly recall some of the most relevant objects to our later considerations.
	
	\subsection{Schwartz Distributions}
	
	Throughout this paper, we shall denote by $\R^n$ the space of points $(x_1, \ldots, x_n)$ where $x_1, \ldots x_n \in \R$. On occasion, we shall use the multi-index notation. Namely, we define for $\alpha \in \mathbb{N}_{0}^{n}$ with $\alpha = (\alpha_1, \ldots, \alpha_n)$ the following quantities:
	\begin{equation}
		x^{\alpha} := \prod_{j=1}^{n} x_{j}^{\alpha_j}, \quad \partial^{\alpha} f(x) = \partial_{x_1}^{\alpha_1} \ldots \partial_{x_n}^{\alpha_n} f(x)
	\end{equation}
	Moreover, we shall sometimes use:
	$$| \alpha | = \sum_{j=1}^{n} \alpha_j$$
	Using these definitions, we are able to introduce the Schwartz semi-norms: For any $\alpha, \beta \in \mathbb{N}_{0}^{n}$ and smooth function $f: \R^n \to \C$, we set:
	$$p_{\alpha, \beta}(f) := \sup_{x \in \R^n} \big{|} x^{\alpha} \partial^{\beta} f(x) \big{|}$$
	Resulting from these, we have:
	$$\forall n \geq 0: p_{n}(f) := \sup_{| \alpha |, | \beta | \leq n} p_{\alpha, \beta}(f)$$
	It is straightforward to see that these define seminorms. The \emph{Schwartz functions} are now defined to be the following subset of all smooth functions:
	\begin{equation}
		\mathcal{S}(\R^{n}) := \Big{\{} f \in C^{\infty}(\R^n)\ \big{|}\ \forall n \geq 0: p_{n}(f) < +\infty \Big{\}}
	\end{equation}
	The finiteness of all seminorms implies that any Schwartz function decreases faster than any polynomial grows. This is the reason why these functions are sometimes also called \emph{rapidly decreasing}. The space $\mathcal{S}(\R^n)$ can also be equipped with a topology stemming from the metric:
	\begin{equation}
		d(f,g) := \sum_{j=0}^{+\infty} 2^{-j} \min \{ 1, p_{j}(f-g) \}
	\end{equation}
	This renders $\mathcal{S}(\R^n)$ a \emph{Fr\'echet space}, thus ensuring that a variety of good properties carry over to this locally convex topological vectorspace (such as a variant of Banach-Steinhaus' Theorem). Naturally, as is well-known, the Fourier transform $f \mapsto \hat{f}$ defines an isomorphism of $\mathcal{S}(\R^n)$ onto itself and we define the following operation for later use:
	\begin{equation}
		f^{\#}(x) := f(-x), \quad \forall f \in \mathcal{S}(\R^n).
	\end{equation}
	Naturally, we have $f^{\#} \in \mathcal{S}(\R^n)$ as well.\\
	
	Having a topology on $\mathcal{S}(\R^n)$ available enables us to define its topological dual space: The space of \emph{tempered distributions} $\mathcal{S}^{\prime}(\R^n)$ is precisely the collection of all continuous linear maps $T: \mathcal{S}(\R^n) \to \C$. By duality, extending the identities available for the regular distributions $\mathcal{S}(\R^n) \subset \mathcal{S}^{\prime}(\R^n)$, one may also define the Fourier transform and the $\#$-operator on distributions:
	\begin{equation}
		\langle \hat{T}, f \rangle := \langle T, \hat{f} \rangle, \quad \langle T^{\#}, f \rangle := \langle T, f^{\#} \rangle,
	\end{equation}
	leading to distributions $\hat{T}, T^{\#} \in \mathcal{S}^{\prime}(\R^n)$, provided $T$ belongs to this space.

	\section{Off-Diagonal Distributions}
	\label{offdiagdist}
	
	In this section, we shall introduce the notion of off-diagonal Schwartz functions and off-diagonal distributions and present the, for our purposes in the current note, most frequently used operations on these sets like fractional gradients, fractional divergences and convolutions. In a second step, we shall investigate the relationship between the pointwise fractional gradient as seen in Mazowiecka-Schikorra \cite{schikorra} and the distributional fractional gradient we define here and deduce some natural characterisations for certain fractional Sobolev spaces in terms of off-diagonal distributions.
	
	\subsection{Definitions and some Operations}
	
	Let us start by introducing a suitable subspace of $\mathcal{S}(\R^{2n})$ that incorporates properties which are at the same time natural to impose and sufficient to later allow for the existence of fractional divergence operators in a pointwise sense as in Mazowicka-Schikorra \cite[p.6]{schikorra}:
	
	\begin{defi}
		We denote by $\mathcal{S}_{od}(\R^{2n})$ the set of \emph{off-diagonal Schwartz functions} on $\R^{2n}$ given by:
		\begin{equation}
		\label{odschwartz}
			\mathcal{S}_{od}(\R^{2n}) := \Big{\{} f \in \mathcal{S}(\R^{2n})\ \big{|} f(x,y) = - f(y,x) \text{ and } \forall \alpha, \beta \in \mathbb{N}^n: \partial_{x}^\alpha \partial_{y}^\beta f(z,z) = 0, \forall z \in \R^n \Big{\}}
		\end{equation}
		This space shall be considered as a subspace of $\mathcal{S}(\R^{2n})$ equipped with the subspace topology, i.e. the Fr\'echet topology induced by the Schwartz seminorms.
	\end{defi}
	
	Naturally, it is easy to see that $\mathcal{S}_{od}(\R^{2n})$ is even a closed subset of $\mathcal{S}(\R^{2n})$, therefore rendering $\mathcal{S}_{od}(\R^{2n})$ a Fr\'echet space as well. We denote its dual space by $\mathcal{S}^{\prime}_{od}(\R^{2n})$ and refer to this set as the set of \emph{off-diagonal tempered distributions}. Natural examples of functions in $\mathcal{S}_{od}(\R^{2n})$ include smooth, compactly supported functions outside of the diagonal which are anti-symmetrised appropriately. The proof that such functions belong to the set of off-diagonal Schwartz functions is trivial.\\
	
	Before we continue, let us introduce the following spaces in analogy to Mazowiecka-Schikorra \cite[p.6]{schikorra}:
	
	\begin{defi}
		For every $p \in [1,\infty[$, let us define the following space:
		$$L^{p}_{od}(\R^{2n}) := \Big{\{} F \in \mathcal{M}_{od}(\R^{2n})\ \big{|}\ \| F \|_{L^{p}_{od}} < \infty  \Big{\}},$$
		where $\mathcal{M}_{od}(\R^{2n})$ denotes the measurable off-diagonal maps and we used:
		$$\| F \|_{L^{p}_{od}}^{p} := \int_{\R^n} \int_{\R^n} | F(x,y) |^p \frac{dx dy}{|x-y|^n}$$
	\end{defi}
	
	We implicitly assume anti-symmetry of the function $F$ here (i.e. $F(x,y) = - F(y,x)$), as the even parts would always vanish when evaluated on $\mathcal{S}_{od}(\R^{2n})$ as an off-diagonal distribution in $\mathcal{S}^{\prime}_{od}(\R^{2n})$. Naturally, one may reasonably expect that a meaningful definition of off-diagonal distribution should include all $L^{p}_{od}$-spaces. Indeed, we may identify $F \in L^{p}_{od}(\R^{2n})$ with an off-diagonal tempered distribution in the following way:
	$$\langle F, G \rangle := \int_{\R} \int_{\R} F(x,y) G(x,y) \frac{dx dy}{| x-y |^n}, \quad \forall G \in \mathcal{S}_{od}(\R^{2n})$$
	Continuity of the distribution in the natural Fr\'echet topology becomes obvious using H\"older's inequality and the behaviour of $G$ along $x = y$ to deduce sufficient local integrability and a rapid decay of $G$. We emphasise that the vanishing derivatives of arbitrary order along the diagonal are helpful in this argument.\\
	
	For the remainder of this introductory section on $\mathcal{S}_{od}(\R^{2n})$ and its dual $\mathcal{S}^{\prime}_{od}(\R^{2n})$, we shall attempt to generalise two important operations for later applications: Smoothing by convolution and the fractional divergence as introduced in Mazowiecka-Schikorra \cite[p.6]{schikorra}. We shall start with the former, as the definition is very natural:
	
	\begin{lem}
	\label{convolution}
		Let $\varphi \in \mathcal{S}(\R^n)$ and $G \in \mathcal{S}_{od}(\R^{2n})$ be given. Let us define the following:
		\begin{equation}
		\label{defconv}
			\forall x,y \in \R^n: \varphi \ast G(x,y) := \int_{\R^n} \varphi(z) G(x-z, y-z) dz
		\end{equation}
		Then:
		$$\varphi \ast G \in \mathcal{S}_{od}(\R^{2n}),$$
		and in addition, the map $G \mapsto \varphi \ast G$ is continuous.
	\end{lem}
	
	  \begin{proof}
	  	Due to the definition in \eqref{defconv} and the rapid decay of both $\varphi$ and $G$, it is obvious that:
		$$\varphi \ast G \in C^{\infty}(\R^{2n}),$$
		and we may pull derivatives inside the integral. Next, let us consider the behaviour of the convolution along the diagonal. Thus, if $x = y \in \R^n$, we immediately see:
		\begin{align}
		\label{convolutionlemmavanishing}
			\varphi \ast G(x,x)	&= \int_{\R^n} \varphi(z) G(x-z, x-z) dz \notag \\
							&= \int_{\R^n} \varphi(z) \cdot 0 dz = 0,
		\end{align}
		where we used that $G$ vanishes along the diagonal, an immediate consequence of $G(x,x) = - G(x,x)$ thanks to anti-symmetry. Moreover, we also observe:
		\begin{align}
			\varphi \ast G(y,x)	&= \int_{\R^n} \varphi(z) G(y-z, x-z) dz \notag \\
							&= - \int_{\R^n} \varphi(z) G(x-z, y-z) dz \notag \\
							&= - \varphi \ast G (x,y)
		\end{align}
		So the convolution as defined in \eqref{defconv} satisfies the anti-symmetry property as required for an off-diagonal Schwartz function. In another step, we would like to verify that the partial derivatives vanish along the diagonal. To achieve this, let us observe that, as we are able to pull derivatives of \eqref{defconv} into the integral, we may notice:
		$$\partial_{x}^\alpha \partial_{y}^\beta \big{(} \varphi \ast G \big{)} (z,z) = \big{(} \varphi \ast \partial_{x}^\alpha \partial_{y}^\beta G \big{)} (z,z) = 0,$$
		where in the last equation we employed the same argument as in \eqref{convolutionlemmavanishing}. As $\alpha, \beta \in \mathbb{N}^n$ were arbitrary multi-indices, the statement follows.\\
		
		It remains to show that the convolution $\varphi \ast G$ has rapid decay. In fact, continuity of the convolution operator $G \mapsto \varphi \ast G$ will be an immediate corollary of our considerations here, as all estimates we will be using shall only depend on Schwartz seminorms. Notice that we have to prove that for any $\gamma, \delta, \alpha, \beta \in \mathbb{N}^n$:
		\begin{equation}
		\label{schwarzsup}
			\sup_{x,y \in \R^n} \big{|} x^\gamma y^\delta \partial_{x}^{\alpha} \partial_{y}^{\beta} \big{(} \varphi \ast G \big{)} (x,y) \big{|} < \infty
		\end{equation}
		One observes that it suffices to consider the case $\alpha = \beta = 0$, since any derivatives of $G$ are still Schwartz functions continuously depending on $G$, hence allowing an inductive argument to deal with general $\alpha$ and $\beta$. We may consider the case $| x | \leq | y |$, the other one may easily be treated in a completely analogous manner by anti-symmetry. Let us split the integral in the definition \eqref{defconv} as follows:
		$$\int_{\R^n} \varphi(z) G(x-z, y-z) dz = \int_{|z| \leq |y|/2} \varphi(z) G(x-z, y-z) dz + \int_{|z| > |y|/2} \varphi(z) G(x-z, y-z) dz$$
		For the second summand, we merely estimate $G$ by its supremum und use the following estimate:
		$$| \varphi(z) | \leq \frac{C_{N, \varphi}}{(1 + |z|)^N},$$
		due to $\varphi \in \mathcal{S}(\R^{n})$. Integration now shows that, as we may choose $N$ arbitrarily large (depending on $\gamma, \delta$ involving estimates of Schwartz seminorms), the contribution of this summand has bounded supremum in \eqref{schwarzsup}:
		\begin{align}
			\Big{|} x^{\gamma} y^{\delta} \int_{|z| > |y|/2} \varphi(z) G(x-z, y-z) dz \Big{|}	&\leq |y|^{| \gamma | + | \delta |} \int_{|z| > |y|/2} | \varphi(z) | \| G \|_{L^{\infty}} dz \notag \\
																		&\leq C_{N, \varphi} \| G \|_{L^{\infty}} | y |^{|\gamma| + | \delta |} \int_{|z| > |y|/2} \frac{1}{(1+|z|)^{N}} dz \notag \\
																		&\leq {C}_{N, n, \varphi} \| G \|_{L^{\infty}} | y |^{|\gamma| + | \delta |} \int_{|y|/2}^{+\infty} \frac{r^{n-1}}{(1+r)^{N}} dr \notag \\
																		&\leq 2^{|\gamma| + | \delta |} {C}_{N, n, \varphi} \| G \|_{L^{\infty}} \int_{0}^{+\infty} \frac{r^{n+|\gamma|+|\delta| -1}}{(1+r)^{N}} dr \notag \\
																		&\leq \tilde{C}_{N, n, \gamma, \delta, \varphi} \| G \|_{L^{\infty}} < +\infty.
		\end{align}
		Consequently, it remains to consider the first summand. In this case, we observe that due to the fact that $G$ is Schwartz:
		\begin{equation}
		\label{helferungleichungyz}
			| G(x,y) | \leq \frac{C_{N}}{(1+|x| + |y|)^N} \leq \frac{C_{N}}{(1+|y|)^N},
		\end{equation}
		thus again choosing $N$ sufficiently large, and observing that:
		\begin{equation}
		\label{helferungleichungyz2}
			\frac{1}{1+|y-z|} \leq \frac{1}{1 + | y | - | z |} \leq \frac{1}{1 + \frac{|y|}{2}} \lesssim \frac{1}{1 + |y|},
		\end{equation}
		for all $|z| \leq |y|/2$, we deduce that also the contribution of this summand to \eqref{schwarzsup} is bounded as well, using \eqref{helferungleichungyz}, \eqref{helferungleichungyz2}:
		\begin{align}
			\Big{|} x^{\gamma} y^{\delta} \int_{|z| \leq |y|/2} \varphi(z) G(x-z, y-z) dz \Big{|}	&\leq | y |^{|\gamma|+|\delta|} \int_{|z| \leq |y|/2} \| \varphi \|_{L^{\infty}} \frac{C_N}{(1+|y-z|)^{N}} dz \notag \\
																		&\leq \tilde{C}_{N} \| \varphi \|_{L^{\infty}} | y |^{| \gamma| + |\delta|} \int_{|z| \leq |y|/2} \frac{1}{(1+|y|)^{N}} dz \notag \\
																		&\leq \frac{\tilde{C}_{N}}{2^{n}} \| \varphi \|_{L^{\infty}} \frac{|y|^{n+|\gamma|+|\delta|}}{(1+|y|)^{N}} < + \infty.
		\end{align}
		Therefore, $\varphi \ast G$ is Schwartz and thus, we have proven the desired continuity result by explicitly using the Schwartz seminorms of $G$ in the estimates above.
	  \end{proof}
	
	It should be noted, that the definition of convolution given above naturally extends to convolutions with functions $F \in L^{1}_{od}(\R^{2n})$ and that one can easily verify using a direct computations that:
	\begin{equation}
	\label{l1young}
		\| \varphi \ast F \|_{L^{1}_{od}} \leq \| \varphi \|_{L^1} \cdot \| F \|_{L^{1}_{od}}
	\end{equation}
	In fact, one even obtains the following general inequality by using Minkowski's inequality for all $p \in [1,+\infty]$:
	\begin{equation}
	\label{generall1young}
		\| \varphi \ast F \|_{L^{p}_{od}} \leq \| \varphi \|_{L^1} \cdot \| F \|_{L^{p}_{od}}
	\end{equation}
	This generalisation will be of great interest later on when we will be regularising fractional gradients in an appropriate way.\\
	
	Next, we turn to the study of fractional gradients for tempered distributions. These considerations are based on the following property of the fractional divergence:
	
	\begin{lem}
	\label{fracdiv}
		Let $G \in \mathcal{S}_{od}(\R^{2n})$ be given. Then we have:
		$$\div_{1/2} G \in \mathcal{S}(\R^n),$$
		where we use the following definition for the $1/2$-divergence:
		\begin{equation}
		\label{definitionhalfdivergenceforschartz}
			\forall x \in \R^n: \div_{1/2} G(x) := \int_{\R^n} \frac{2G(x,y)}{|x-y|^{1/2}} \frac{dy}{| x-y |^n} = \int_{\R^n} \frac{G(x,y) - G(y,x)}{|x-y|^{1/2}} \frac{dy}{| x-y |^n}
		\end{equation}
		In addition, the map $G \mapsto \div_{1/2} G$ is continuous with respect to the Schwartz topology.
	\end{lem}
	
	  \begin{proof}
	  	We notice that $\div_{1/2} G$ is well-defined in a pointwise sense due to all derivatives in $y$ direction vanishing up to arbitrary order at $y = x$ and Taylor approximation. To deduce smoothness, let us notice that by a change of variables:
		$$\int_{\R^n} \frac{G(x,y)}{|x-y|^{1/2}} \frac{dy}{| x-y |^n} = \int_{\R^n} \frac{G(x, x-z)}{|z|^{n + 1/2}} dz$$
		It is now obvious that by the properties of off-diagonal Schwartz functions, we may deduce as previously for the convolution in \eqref{defconv}:
		$$\div_{1/2} G \in C^{\infty}(\R^n),$$
		as we can pull the derivatives into the integral due to the integrability properties of $G$ and all its derivatives. One notices that the vanishing of all partial derivatives along the diagonal is strictly required here. This now provides retrospective motivation for the definition of the space $\mathcal{S}_{od}(\R^{2n})$.\\
		
		Thus, it remains to verify that $\div_{1/2} G$ is rapidly decaying. The estimates we will establish merely depend on the Schwartz seminorms of $G$ and thus also prove continuity by definition of the Schwartz topology. Once again, we would like to estimate:
		$$\sup_{x \in \R} \big{|} x^\gamma \partial_{x}^\alpha \div_{1/2} G (x) \big{|} < \infty,$$
		where $\gamma, \alpha \in \mathbb{N}^n$ are arbitrary. As in the proof of Lemma \ref{convolution}, it suffices to consider $\alpha = 0$, as the partial derivative can be pulled into the intergal, leading to estimates of $1/2$-divergences for some partial derivatives $G$. These would then yield the corresponding estimates with respect to $G$, so $\alpha = 0$ suffices. Again, let us split the intergal:
		$$\int_{\R^n} \frac{G(x, x-z)}{|z|^{n + 1/2}} dy = \int_{|z| \leq |x|/2} \frac{G(x, x-z)}{|z|^{n + 1/2}} dy + \int_{|z| > |x|/2} \frac{G(x, x-z)}{|z|^{n + 1/2}} dy$$
		The second summand can be estimated by using the Schwartz decay of $G$:
		$$| G(x,x-z) | \leq \frac{C_{N}}{(1 + | x | + | x-z |)^{N}} \leq \frac{C_N}{(1+|x|)^{N}},$$
		and therefore we obtain:
		$$\Big{|} \int_{|z| > |x|/2} \frac{G(x, x-z)}{|z|^{n + 1/2}} dy \Big{|} \leq \frac{C_{N,n}}{(1+|x|)^{N}} \frac{1}{| x |^{1/2}}$$
		So the estimate should only be considered for $|x| \geq 1$. However, it actually suffices to bound on this set, as can be easily seen. So we merely have to consider the first summand. There, we can bound:
		$$| G(x,x-z) | \leq \frac{C_N}{(1+|x|)^{N}} \cdot | z |,$$
		by using Taylorapproximation at $(x,x)$ and since the integral of $| z |^{-n + 1/2}$ is bounded by $| x |^{1/2}$, the resulting expression is bounded in a similar way as before. So $\div_{1/2} G \in \mathcal{S}(\R^n)$.
	  \end{proof}
	  
	  It should be noted here, that there is no particular reason to restrict ourselves to $s = 1/2$, except for our a-priori interest in this particular case. Indeed, the following generalisation follows precisely along the same lines, we may only have to use higher order Taylor approximations:
	  
	  \begin{lem}
	\label{sfracdiv}
		Let $G \in \mathcal{S}_{od}(\R^{2n})$ be given. Then we have for all $s \in \R$:
		$$\div_{s} G \in \mathcal{S}(\R^n),$$
		where we use the following definition for the $s$-divergence:
		\begin{equation}
		\label{definitiongeneralfractionaldivergenceforarbs}
			\forall x \in \R^n: \div_{s} G(x) := \int_{\R^n} \frac{2G(x,y)}{|x-y|^{s}} \frac{dy}{| x-y |^n} = \int_{\R^n} \frac{G(x,y) - G(y,x)}{|x-y|^{s}} \frac{dy}{| x-y |^n}
		\end{equation}
		In addition, the map $G \mapsto \div_{s} G$ is continuous.
	\end{lem}
	
	The adaption of the proof of Lemma \ref{fracdiv} is left to the reader.\\
	
	  The significance of Lemma \ref{convolution} lies in the immediate connection to smoothing operations which we shall exploit later in order to restrict our attention to more regular distributions. This is useful when we try to establish estimates. On the other hand, Lemma \ref{fracdiv} actually provides a suitable notion of fractional $1/2$-divergence on certain functions which will enable us to define the fractional $1/2$-gradient $d_{1/2}$ for tempered distributions $u \in \mathcal{S}^\prime(\R)$. Namely, we define: %for any such $u$:
	  \begin{defi}
	  \label{definitionfracgradondist}
	  	Given any $u \in \mathcal{S}^{\prime}(\R^n)$, we define $d_{1/2} u \in \mathcal{S}^{\prime}(\R^{2n})$ in the following way:
	 	\begin{equation}
	 	\label{dualityfracgrad}
	  		\forall G \in \mathcal{S}_{od}(\R^{2n}): \quad \langle d_{1/2} u, G \rangle := \langle u, \div_{1/2} G \rangle
		\end{equation}
		Here, $\langle \cdot, \cdot \rangle$ denotes the usual pairing between test functions and distributions.
	\end{defi}
	Clearly, Lemma \ref{fracdiv} proves that $d_{1/2} u \in \mathcal{S}^{\prime}_{od}(\R^{2n})$ by the continuity of the fractional divergence. Additionally, we can define $d_{s}$ for any $s \in \R$ much in the same way:
	\begin{equation}
	\label{definitiondistfracgradondist}
		\forall u \in \mathcal{S}^\prime (\R^n), \forall G \in \mathcal{S}^{\prime}_{od}(\R^{2n}): \quad \langle d_{s} u, G \rangle := \langle u, \div_{s} G \rangle,
	\end{equation}
	and as before, due to the continuity established in Lemma \ref{sfracdiv}, we have $d_{s} u \in \mathcal{S}^{\prime}_{od}(\R^{2n})$.\\
	
	Finally, let us present the following computation that will be crucial to our smoothing procedure and shows that the operations introduced are well-behaved with respect to convolution. For any $u \in \mathcal{S}^{\prime}(\R^n)$ and $G \in \mathcal{S}_{od}(\R^{2n})$, we have the following sequence of identities (where convolution on the distribution side is defined by duality as usual) for any $\varphi \in \mathcal{S}(\R^n)$:
	\begin{align}
	\label{proofofcomputationsofcommconvandfrac}
		\langle \varphi \ast d_{1/2}u, G \rangle	&= \langle d_{1/2}u, \varphi^{\#} \ast G \rangle \notag \\
										&= \langle u, \div_{1/2} \big{(} \varphi^{\#} \ast G \big{)} \rangle \notag \\
										&= \langle u, \varphi^{\#} \ast \div_{1/2} G \rangle \notag \\
										&= \langle \varphi \ast u, \div_{1/2} G \rangle \notag \\
										&= \langle d_{1/2} \big{(} \varphi \ast u \big{)} , G \rangle,
	\end{align}
	where we used:
	\begin{align}
	\label{proofofcomputationsofcommconvandfrac2}
		\div_{1/2} \big{(} \varphi^{\#} \ast G \big{)}(x) 	&= \int_{\R^n} \frac{2\varphi^{\#} \ast G(x,y)}{| x-y |^{1/2}} \frac{dy}{| x-y |^n} \notag \\
											&= \int_{\R^n} \int_{\R^n} \frac{\varphi^{\#}(z) 2G(x-z,y-z)}{| x-y |^{1/2}} dz \frac{dy}{| x-y |^n} \notag \\
											&= \int_{\R^n} \int_{\R^n} \frac{\varphi^{\#}(z) 2G(x-z,y-z)}{| x-y |^{1/2}} \frac{dy}{| x-y |^n} dz \notag \\
											&= \int_{\R^n} \varphi^{\#}(z) \int_{\R^n} \frac{2G(x-z,\tilde{y})}{| x-z-\tilde{y} |^{1/2}} \frac{d\tilde{y}}{| x-z-\tilde{y} |^n} dz \notag \\
											&= \int_{\R^n} \varphi^{\#}(z) \div_{1/2} G(x-z) dz \notag \\
											&=\big{(} \varphi^{\#} \ast \div_{1/2} G \big{)}(x)
	\end{align}
	We would like to mention that $\varphi^{\#}(x) := \varphi(-x)$. Thus, convolution and fractional divergence/gradient behave well with respect to each other (as expected from the local case) and so do fractional gradient and convolution. Again, the computations above generalize to fractional gradients $d_{s}$ and divergences $\div_s$ for arbitrary $s \in \mathbb{R}$, as the smoothness index $s=1/2$ did not play a significant role in the computations above:
	\begin{equation}
		\varphi \ast d_{s}u = d_{s} \big{(} \varphi \ast u \big{)}, \quad \forall \varphi \in \mathcal{S}(\R^n), \forall u \in \mathcal{S}^{\prime}(\R^n)
	\end{equation}
	The computations and comments above can be summarised as follows:
	\begin{lem}
	The following identity holds true for all real numbers $s \in \mathbb{R}$ as well as all $u \in \mathcal{S}^{\prime}(\R^n), G \in \mathcal{S}_{od}(\R^{2n})$ and $\varphi \in \mathcal{S}(\R^n)$:
	\begin{equation}
	\label{fracgradconv}
		\varphi \ast d_{s}u = d_{s} \big{(} \varphi \ast u \big{)}, \quad \varphi \ast \div_{s} G = \div_{s} \left( \varphi \ast G \right).
	\end{equation}
	\end{lem}
	The proof is provided by the computations in \eqref{proofofcomputationsofcommconvandfrac} and \eqref{proofofcomputationsofcommconvandfrac2} above.\\
	
	It is also important to keep the following in mind for all $u \in \mathcal{S}(\R^n) \subset \mathcal{S}^\prime(\R^n)$ and $G \in \mathcal{S}_{od}(\R^{2n})$:
	\begin{align}
	\label{auxilliary}
		\langle d_{1/2} u, G \rangle	&= \langle u, \div_{1/2} G \rangle \notag \\
								&= \int_{\R^n} u(x) \div_{1/2} G(x) dx \notag \\
								&= \int_{\R^n} u(x) \int_{\R^n} \frac{2 G(x,y)}{| x-y |^{n+1/2}} dy dx \notag \\
								&= 2 \int_{\R^n} \int_{\R^n} \frac{u(x) G(x,y)}{| x-y |^{1/2}} \frac{dy dx}{| x-y |^n} \notag \\
								&= \int_{\R^n} \int_{\R^n} \frac{u(x) (G(x,y) - G(y,x))}{| x-y |^{1/2}} \frac{dy dx}{| x-y |^n} \notag \\
								&= \int_{\R^n} \int_{\R^n} \frac{(u(x) - u(y)) G(x,y)}{| x-y |^{1/2}} \frac{dx dy}{| x-y |^n} \notag \\
								&= \int_{\R^n} \int_{\R^n} d_{1/2} u(x,y) \cdot G(x,y) \frac{dx dy}{| x-y |^n},
	\end{align}
	where we used the pointwise definition of the $1/2$-gradient as in Mazowiecka-Schikorra \cite{schikorra} in the last equation. This shows that our definition of the (distributional) fractional gradient is consistent with the (pointwise) definition in Mazowiecka-Schikorra \cite[p.6]{schikorra}. Moreover, by approximation, the identity above can be extended to all functions for which the identity makes sense, due to the continuity with respect to the usual fractional Gagliardo-Sobolev norms (we shall explain this in more detail in the upcoming subsections). Finally, the computation in \eqref{auxilliary} naturally generalises to arbitrary $d_{s}$ and $\div_{s}$ without additional effort.\\

	\subsection{The Spaces $H^{s}(\R)$: Connections to the Fractional Gradient on Distributions}
	
	In Mazowiecka-Schikorra \cite{schikorra}, the fractional gradient was introduced in a pointwise sense by the formula:
	\begin{equation}
	\label{pointwisefracgrad}
		d_{s} u(x,y) := \frac{u(x) - u(y)}{| x-y |^s}, \quad \forall u \in \mathcal{M}(\R), \forall s \in \R_{+}.
	\end{equation}
	This expression connects via some Wente-type estimates proven in Mazowiecka-Schikorra \cite{schikorra} to integrability by compensation-phenomena and the regularity of fractional harmonic maps. It is a natural question to ask whether the fractional gradient in the distributional sense as defined in \eqref{dualityfracgrad} and the pointwise gradient in \eqref{pointwisefracgrad} agree. Such questions are also connected to characterisations of Sobolev-type spaces like the following found in Prats-Saksman \cite[Theorem 1.2]{prats} and Schikorra-Wang \cite[Theorem 1.4]{schiwang}:
	
	\begin{thm}
	\label{schiwangthm1.4}
		Let $s \in (0,1)$, $p,q \in ]1, \infty[$ and $f \in L^p(\R)$. Then, if we denote by $\dot{F}^{s}_{p,q}(\R)$ the usual homogeneous Triebel-Lizorkin function spaces:
		\begin{itemize}
			\item[(i)] We know $\dot{W}^{s, (p,q)}(\R) \subset \dot{F}^{s}_{p,q}(\R)$ together with:
			\begin{equation}
				\| f \|_{\dot{F}^{s}_{p,q}(\R)} \lesssim \| f \|_{\dot{W}^{s, (p,q)}(\R)}
			\end{equation}
			
			\item[(ii)] If $p > \frac{q}{1 + sq}$, then we also have the converse inclusion together with:
			\begin{equation}
			\label{secondpartofthm2.1}
				\| f \|_{\dot{W}^{s, (p,q)}(\R)} \lesssim \| f \|_{\dot{F}^{s}_{p,q}(\R)}
			\end{equation}
		\end{itemize}
		The constants depend on $s, p, q$.
	\end{thm}
	
	The same result continues to hold true for $\R^n$ is we replace the condition $p > \frac{q}{1+sq}$ by $p > \frac{nq}{n+sq}$, see Prats-Saksman \cite[Theorem 1.2]{prats}. To be precise, the following definition for the spaces $\dot{W}^{s,(p,q)}(\R^n)$ is used: For any $f: \R^n \to \R$:
	$$\mathcal{D}_{s,q}(f)(x) := \left( \int_{\R} \frac{| f(x) - f(y) |^q}{| x-y |^{sq}} \frac{dy}{| x-y |^n} \right)^{1/q} = \left( \int_{\R} | d_{s} f(x,y) |^q \frac{dy}{| x-y |^n} \right)^{1/q},$$
	for all $1 \leq q < \infty$ and $0 < s < 1$ and moreover:
	$$\| f \|_{\dot{W}^{s,(p,q)}(\R^n)} := \| \mathcal{D}_{s,q}(f)(x) \|_{L^{p}(\R^n)},$$
	for every $1 \leq p \leq \infty$. If $p = q$, these spaces correspond to the usual homogeneous Gagliardo-Sobolev spaces $\dot{W}^{s,p}(\R^n)$.\\
	
	The key point of Theorem \ref{schiwangthm1.4} is that we may characterise the spaces $\dot{H}^{s}(\R^n) \simeq \dot{F}^{s}_{2,2}(\R^n)$ and more generally the Bessel potential spaces $\dot{F}^{s}_{p,2}(\R^n)$ for $s \in (0,1)$ by means of the integrability of the pointwise fractional gradient \eqref{pointwisefracgrad}, namely:
	$$d_{s} u \in L^{2}_{od}(\R^{2n}) \Leftrightarrow u \in \dot{H}^{s}(\R^n), \quad \forall s \in (0,1)$$
	A natural question to pose is then whether these characterisations carry over to the distributional setup we introduced in the previous subsection.\\
	
	We have previously seen, thanks to \eqref{auxilliary}, that the pointwise definition in \eqref{pointwisefracgrad} and the distributional one in \eqref{dualityfracgrad} agree at least for $u \in \mathcal{S}(\R)$. This is an indication that our definitions of a distribution and the distributional fractional gradient are reasonable and should allow for an extension of the corresponding notion of fractional gradient in Mazowiecka-Schikorra \cite{schikorra}.
	
	Let us now consider a function $u \in {H}^{s}(\R^n)$. We may now either approximate $u$ by functions $u_n \in \mathcal{S}(\R^n)$ converging in ${H}^s (\R^n)$ to $u$ and deduce from the identity \eqref{auxilliary} that $d_s u$ pointwise and in the distributional sense are the same. Alternatively, the following computation proves this as well:
	\begin{align}
	\label{compofdistgradispointwisefracgrad}
		\forall G \in \mathcal{S}_{od}(\R^{2n}): \langle d_s u, G \rangle 	&= \langle u, \div_s G \rangle \notag \\
														&= \int_{\R^n} u(x) \int_{\R^n} \frac{2 G(x,y)}{| x-y |^s} \frac{dy}{| x-y |^n} dx \notag \\
														&= \iint_{\R^{2n}} \frac{2 u(x) G(x,y)}{| x-y |^s} \frac{dy dx}{| x-y |^n} \notag \\
														&= \iint_{\R^{2n}} \frac{u(x) - u(y)}{| x-y |^s} G(x,y) \frac{dy dx}{| x-y |^n} + \iint_{\R^{2n}} \frac{u(x) + u(y)}{| x-y |^s} G(x,y) \frac{dy dx}{| x-y |^n} \notag \\
														&= \iint_{\R^{2n}} \frac{u(x) - u(y)}{| x-y |^s} G(x,y) \frac{dy dx}{| x-y |^n} \notag \\
														&= \iint_{\R^{2n}} d_{s} u (x,y) G(x,y) \frac{dy dx}{| x-y |^n}.
	\end{align}
	Here, we used:
	$$2 u(x) = \left( u(x) + u(y) \right) + \left( u(x) - u(y) \right),$$
	as well as integrability and fast decay of $G(x,y)$ to ensure integrability of $|x-y|^{-s} (u(x) + u(y)) G(x,y)$ (using Sobolev embeddings). Thus, the pointwise fractional gradient \eqref{pointwisefracgrad} and the distributional one \eqref{dualityfracgrad} agree for all $u \in H^{s}(\R)$. Stated as a lemma, we find:
	
	\begin{lem}
		If $s \in (0,1)$ and $u \in \dot{H}^{s}(\R^n)$, then:
		$$\langle d_s u, G \rangle = \iint_{\R^{2n}} d_{s} u(x,y) G(x,y) \frac{dy dx}{| x-y |^n},$$
		for all $G \in \mathcal{S}_{od}(\R^{2n})$. Hence, the distributional fractional gradient actually agrees with the pointwise definition in \eqref{pointwisefracgrad}.
	\end{lem}
	
	So far, we have been deducing simpler formulas for $d_s u$ provided $u$ is sufficiently regular, i.e. $u \in \dot{H}^{s}(\R^n)$. Naturally, we would also be interested in results going the other direction, therefore using statements about the integrability of the distributional fractional gradient $d_s u$ to obtain regularity of $u$. To start our investigations, let us assume that $s \in (0,1)$ and that we have:
	\begin{equation}
	\label{distgradinl2}
		d_s u \in L^{2}_{od}(\R^{2n}),
	\end{equation}
	i.e. we assume that there exists a $F \in L^{2}_{od}(\R^{2n})$, such that:
	$$\forall G \in \mathcal{S}_{od}(\R^{2n}): \quad \langle d_s u, G \rangle = \iint_{\R^{2}n} F(x,y) G(x,y) \frac{dy dx}{| x-y |^n}.$$
	Such a condition could be formulated in an equivalent way as the existence of a constant $C > 0$, such that for every $G \in L^{2}_{od}(\R^{2n})$, it holds:
	$$\big{|} \langle d_s u, G \rangle \big{|} \leq C \| G \|_{L^{2}_{od}(\R^{2n})}.$$
	To arrive at a regularity statement, the key idea is to approximate $u$ by smooth functions $u_{k}$ of at most polynomial growth, such that:
	$$u_k \to u \text{ in } \mathcal{S}^{\prime}(\R^n), \text{ as } k \to +\infty,$$
	and such that:
	$$\| d_{s} u_k \|_{L^{2}_{od}(\R^{2n})} \leq C \| F \|_{L^{2}_{od}(\R^{2n})}, \quad \forall k \in \mathbb{N},$$
	where $C > 0$ is independent of $k \in \mathbb{N}$. For details on how such an approximating sequence can be obtained, we refer to the proof of Theorem \ref{l1est} below, for now it suffices to know that such approximating sequences do exist and that they may be constructed using convolutions by a smoothing kernel. Now, if $u_k$ is smooth and of at most polynomial growth, we see that:
	$$\iint_{\R^{2n}} \Big{|} \frac{u_k (x) \pm u_k (y)}{| x-y |^s} G(x,y) \Big{|} \frac{dy dx}{| x-y |^n} < + \infty, \quad \forall G \in \mathcal{S}_{od}(\R^2), \forall k \in \mathbb{N}.$$
	Therefore, arguing as in \eqref{compofdistgradispointwisefracgrad} or \eqref{auxilliary}, we find:
	$$\langle d_{s} u_k, G \rangle = \iint_{\R^{2n}} \frac{u_k (x) - u_k (y)}{| x-y |^s} G(x,y) \frac{dy dx}{| x-y |^n}, \quad \forall G \in \mathcal{S}_{od}(\R^{2n}), \forall k \in \mathbb{N}.$$
	Since $u_k$ is smooth and bounded by polynomials, we clearly have $d_s u_k \in L^{2}_{od, loc}(\R^2)$. Therefore, we may deduce, by choosing $G$ to be supported in compact subsets:
	$$d_s u_k = F \text{ on every compact subset } \mathbb{R}^{2n} \setminus \{ (x,y) \in \R^{2n}\ |\ x = y \},$$
	which also proves directly:
	$$d_s u_k \in L^{2}_{od}(\R^{2n}),$$
	i.e. the pointwise fractional gradient $d_s u_k$ is actually in $L^{2}_{od}(\R^{2n})$ and satisfies:
	$$\| d_s u_k \|_{L^{2}_{od}(\R^{2n})} \leq C \| F \|_{L^{2}_{od}(\R^{2n})},$$
	for all $k \in \mathbb{N}$. By Theorem \ref{schiwangthm1.4}, we thus find:
	$$u_k \in \dot{H}^{s}(\R^n) \text{ and } \ \| u_k \|_{\dot{H}^{s}(\R^n)} \leq \tilde{C} \| F \|_{L^{2}_{od}(\R^{2n})}, \quad \forall k \in \mathbb{N}.$$
	This implies that there is a weakly converging subsequence of $(u_k)$ converging to some $\tilde{u} \in \dot{H}^{s}(\R^n)$. Now, this shows that:
	$$\langle d_s u, G \rangle = \lim_{n \to +\infty} \langle d_s u_k, G \rangle = \langle d_s \tilde{u}, G \rangle, \quad \forall G \in \mathcal{S}_{od}(\R^{2n}),$$
	showing that:
	$$d_s \left( u - \tilde{u} \right) = 0.$$
	In the next subsection in Theorem \ref{l1est}, we shall see that this implies:
	$$u - \tilde{u} = C,$$
	for some constant $C \in \mathbb{R}$. Thus, we also have:
	$$u = \tilde{u} + C \in \dot{H}^{s}(\R^n).$$
	To summarise, we have shown:
	
	\begin{lem}
		Let $s \in (0,1)$, $u \in \mathcal{S}^{\prime}(\R^n)$ and assume that there exists a constant $C > 0$, such that:
		\begin{equation}
		\label{lemmaestforoperatornorm}
			\left| \langle d_s u, G \rangle \right| \leq C \| G \|_{L^{2}_{od}(\R^{2n})}, \quad \forall G \in \mathcal{S}_{od}(\R^{2n}).
		\end{equation}
		Then:
		$$u \in \dot{H}^{s}(\R^n),$$
		and we know that $\| u \|_{\dot{H}^{s}(\R^n)}$ can be estimated from above and below by the operatornorm (i.e. the smallest $C$ possible in the inequality \eqref{lemmaestforoperatornorm} above).
	\end{lem}
	
	Using Theorem \ref{schiwangthm1.4} for other values of $p$, we may also deduce corresponding estimates for these cases, namely:
	
	\begin{lem}
		Let $s \in (0,1)$, $1 < p < +\infty$, $u \in \mathcal{S}^{\prime}(\R^n)$ and assume that there exists a constant $C > 0$, such that:
		\begin{equation}
		\label{lemmaestforoperatornormgeneralp}
			\left| \langle d_s u, G \rangle \right| \leq C \left( \int_{\R^{n}} \left( \int_{\R^n} | G(x,y) |^2 \frac{dy}{| x-y |^n} \right)^{q/2} dx \right)^{1/q}, \quad \forall G \in \mathcal{S}_{od}(\R^{2n}),
		\end{equation}
		where:
		$$\frac{1}{p} + \frac{1}{q} = 1.$$
		Then:
		$$u \in \dot{F}^{s}_{p,2}(\R^n),$$
		and we know that $\| u \|_{\dot{F}^{s}_{p,2}(\R^n)}$ can be estimated from above and below by the operatornorm (i.e. the smallest $C$ possible in the inequality \eqref{lemmaestforoperatornorm} above).
	\end{lem}
	
	The proof proceeds precisely the same way, we again refer to the next section for more details concerning the kind of approximations we use. 
	
	\subsection{Fractional Laplacians and Distributional Fractional Gradients}
	
	Lastly, we would like to briefly explore possible connections between the fractional Laplacian and the distributional fractional gradient as introduced in \eqref{definitiondistfracgradondist}. We shall mostly restrict to the case $\R$ for convenience's sake.
	
	As seen in Mazowiecka-Schikorra \cite[p.6]{schikorra}, one has that:
	\begin{equation}
	\label{schikorrafraclaplacianid}
		\div_{s} d_s u = (-\Delta)^s u,
	\end{equation}
	for $u \in \dot{H}^{s}(\R)$ due to the dual definition of the fractional divergence provided there. This identity means, that the following holds:
	$$\int_{\R} \int_{\R} d_{s} f(x,y) d_{s} g(x,y) \frac{dy dx}{| x-y |} = \int_{\R} (-\Delta)^{s} f(x) g(x) dx $$
	However, we see that our definition of the fractional divergence $\div_s$ does not immediately make sense for $d_s u$, as the latter is merely a distribution in $\mathcal{S}_{od}^{\prime}(\R^2)$. Indeed, the fractional divergence we defined in the previous subsections (\eqref{definitiongeneralfractionaldivergenceforarbs}) merely leads to a meaningful object for off-diagonal Schwarz functions. One might wonder, if we may extend the definition of the fractional divergence to a bigger subset of the collection of off-diagonal distributions and do so in such way that an identity similar to \eqref{schikorrafraclaplacianid} continues to hold true for the distributional fractional gradient and (distributional) fractional divergence. Naturally, resolving this issue hinges on the definition of a fractional divergence on off-diagonal distributions. We shall try to provide some ideas into this direction in this subsection.\\
	
	Firstly, assume that $u \in \mathcal{S}(\R), G \in \mathcal{S}_{od}(\R^2)$, then we have:
	$$\langle d_s u, G \rangle = \langle u, \div_{s} G \rangle,$$
	directly due to the definition in \eqref{definitiondistfracgradondist}. Since for $T \in \mathcal{S}_{od}^{\prime}(\R^2)$, we expect $\div_s T$ to be a tempered distribution on $\R$ and $G \in \mathcal{S}_{od}(\R^2) \subset \mathcal{S}_{od}^{\prime}(\R^2)$, this identity suggests that the correct definition should be:
	\begin{equation}
	\label{distributionaldivergencedual}
		\langle \div_s T, \varphi \rangle := \langle T, d_s \varphi \rangle, \quad \forall \varphi \in \mathcal{S}(\R).
	\end{equation}
	However, it is immediately clear that $d_s \varphi$ will not belong to $\mathcal{S}_{od}(\R^2)$, if $\varphi \neq 0$. This can be trivially observed by fixing $y$ appropriately and letting $x \to \pm \infty$ to deduce that the decay is not faster than any given polynomial, if $\varphi \neq 0$. As a result, the meaning behind the pairing \eqref{distributionaldivergencedual} is in general unclear and, in fact, the author is unable to assign a general meaning to \eqref{distributionaldivergencedual} for arbitrary $T \in \mathcal{S}_{od}^{\prime}(\R^2)$. However, for particuarily well-behaved $T$, \eqref{distributionaldivergencedual} can be given a meaning by using an approximation procedure. For this, we need to be able to approximate $d_s \varphi$ by a sequence $(G_k(\varphi)) \subset \mathcal{S}_{od}(\R^2) \subset X \subset \mathcal{S}^{\prime}_{od}(\R^2)$ in a Banach space $X$, which contains all pointwise fractional gradients $d_s \psi$ for $\psi$ a Schwartz function, with respect to the norm $\| \cdot \|_{X}$\footnote{Observe that one may restrict to such spaces $X$ for which $\mathcal{S}_{od}(\R^2)$ is a dense subset. This can be done by going over to the topological closure of the subset $\mathcal{S}_{od}(\R^2)$ in $X$ which continues to possess all necessary properties.} and assume that there exists $C > 0$ such that:
	$$\left| \langle T, G \rangle \right| \leq C \| G \|_{X}, \quad \forall G \in \mathcal{S}_{od}(\R^2).$$
	Then, we may define:
	\begin{equation}
	\label{divergencebyapproximationfordist}
		\langle \div_s T, \varphi \rangle := \lim_{k \to +\infty} \langle T, G_k(\varphi) \rangle,
	\end{equation}
	and this defines $\div_{s} T$ uniquely as a linear operator. To ensure that $\div_s T$ is actually continuous with respect to the locally convex Fr\'echet topology on $\mathcal{S}(\R)$ induced by the Schwartz seminorms, we need to have that the inclusion map $\psi \mapsto d_s \psi$ for all $\psi \in \mathcal{S}(\R)$ is continuous with respect to the norm topology on $X$.
	
	We shall discuss some circumstances under which this is possible:\\
	
	A very simple example is $X = L^{2}_{od}(\R^2)$. Naturally, the space of off-diagonal Schwartz functions lies dense within this space and the map $\psi \mapsto d_s \psi$ for $s \in (0,1)$ is continuous due to the inclusion $\mathcal{S}(\R) \subset \dot{H}^{s}(\R)$. Thus, the existence of approximating sequences $(G_{k}(\varphi))$ in $\mathcal{S}_{od}(\R^2)$ converging to $d_s \varphi$ in $L^{2}_{od}(\R^2)$ for any $\varphi \in \mathcal{S}(\R)$ is ensured. Therefore, we have:
	
	\begin{lem}
	\label{extensionoffracdivtodist}
		Let $T \in \mathcal{S}^{\prime}_{od}(\R^2)$ be such that there exists $C > 0$, such that:
		$$\left| \langle T, G \rangle \right| \leq C \| G \|_{L^{2}_{od}(\R^2)}, \quad \forall G \in \mathcal{S}_{od}(\R^2).$$
		Then $\div_{1/2} T$ can be defined by the formula in \eqref{divergencebyapproximationfordist}.
	\end{lem}
	
	Particular examples of such off-diagonal distributions $T$ shall now be discussed, with an emphasise to their connection with the fractional Laplacians.\\
	
	Namely, let us consider $u \in H^{s}(\R)$. By the previous subsection, we know that $d_s u$ in the distributional sense actually is actually a regular distribution with:
	$$\forall G \in \mathcal{S}_{od}(\R^2): \quad \langle d_s u, G \rangle = \int_{\R} \int_{\R} \frac{u(x) - u(y)}{| x-y |^s} G(x,y) \frac{dy dx}{| x-y |}$$
	We know that $d_s u$ with the pointwise formula above actually lies in $L^{2}_{od}(\R^2)$. As a result, we notice that Lemma \ref{extensionoffracdivtodist} applies in this case by using Cauchy-Schwarz inequality. Therefore, the distribution:
	$$\div_s d_s u \in \mathcal{S}'(\R),$$
	is well-defined by the approximation formula presented above in \eqref{divergencebyapproximationfordist}. This leads us to the following:
	\begin{align}
		\forall \varphi \in \mathcal{S}(\R): \langle \div_s d_s u, \varphi \rangle	&= \lim_{k \to \infty} \langle d_s u, G_k(\varphi) \rangle \notag \\
															&= \lim_{k \to \infty} \int_{\R} \int_{\R} \frac{u(x) - u(y)}{| x-y |^s} G_k(\varphi)(x,y) \frac{dy dx}{| x-y |} \notag \\
															&= \int_{\R} \int_{\R} \frac{u(x) - u(y)}{| x-y |^s} \frac{\varphi(x) - \varphi(y)}{| x-y |^s} \frac{dy dx}{| x-y |} \notag \\
															&= \int_{\R} (-\Delta)^{s/2} u(x) (-\Delta)^{s/2} \varphi(x) dx \notag \\
															&= \langle (-\Delta)^{s} u, \varphi \rangle,
	\end{align}
	i.e. we have that:
	\begin{equation}
		\div_s d_s u = (-\Delta)^{s} u,
	\end{equation}
	as tempered distributions. In particular, if $u \in H^{2s}(\R)$, then $\div_s d_s u$ agrees with the pointwise fractional Laplacian $(-\Delta)^s u$.

	  \section{Bourgain-Brezis Estimates for Fractional Gradients}
	  \label{bbest}
	  
	  Next, we shall focus on estimates that we are able to derive for tempered distributions, provided their fractional gradients are sufficiently integrable. The estimates obtained are reminiscent of the work in \cite{wett} of the author as well as his collaborators.
	  
	  \subsection{The $L^1$-estimate and fractional Sobolev embeddings}
	  
	  The goal of this subsection is to prove the following statement\footnote{For simplicity's sake, we shall formulate it here only for $n=1$. A general statement is supplied at the end of the current subsection.}:
	  
	  \begin{thm}
	  \label{l1est}
	  	Assume that $u \in \mathcal{S}^\prime(\R)$ is such that:
		$$d_{1/2} u = F \in L^{1}_{od}(\R^2) \subset \mathcal{S}^{\prime}_{od}(\R^2),$$
		then there exists a constant $C \in \R$ such that $u - C \in L^{2}(\R)$ together with the estimate:
		\begin{equation}
		\label{mainest}
			\| u - C \|_{L^2} \lesssim \| d_{1/2} u \|_{L^{1}_{od}}
		\end{equation}
	  \end{thm}
	  The main ingredient, besides a suitable approximation and appropriate choice of the notion of distribution (i.e. the one presented in the previous section), is the fractional Sobolev inequality which is proven with specific behaviour of the constants in limiting cases in Bourgain-Brezis-Mironescu \cite{bbm}, but which can also be found in various other references.\\
	  
	  Let now $u$ be as in Theorem \ref{l1est} and denote $d_{1/2}u = F \in L^{1}_{od}(\R^2)$. Using \eqref{fracgradconv}, we can see the following: We may define $\varphi_{\varepsilon}(x) := \frac{1}{\varepsilon} \varphi \big{(} \frac{x}{\varepsilon} \big{)}$ for some $\varphi \in \mathcal{S}(\R)$ and all $\varepsilon > 0$, such that:
	  $$\varphi(x) \geq 0, \quad \int_{\R} \varphi(x) dx = 1$$
	  If we introduce:
	  $$u_{\varepsilon} := \varphi_{\varepsilon} \ast u, \quad \forall \varepsilon > 0,$$
	  then by \eqref{fracgradconv}:
	  $$d_{1/2} u_{\varepsilon} = \varphi_{\varepsilon} \ast d_{1/2} u = \varphi_{\varepsilon} \ast F =: F_{\varepsilon}$$
	  Using \eqref{l1young}, we immediately obtain the uniform estimate:
	  \begin{equation}
	  \label{uniforml1est}
	  	\| d_{1/2} u_{\varepsilon} \|_{L^{1}_{od}} = \| F_{\varepsilon} \|_{L^{1}_{od}} \leq \| F \|_{L^{1}_{od}} = \| d_{1/2} u \|_{L^{1}_{od}}, \quad \forall \varepsilon > 0
	  \end{equation}
	  It is clear that $u_{\varepsilon}$ converges to $u$ in the sense of tempered distributions, if we let $\varepsilon \to 0$. This follows from the observation that for any $\psi \in \mathcal{S}(\R)$, we have:
	  $$\varphi_{\varepsilon}^{\#} \ast \psi \to \psi, \text{ as } \varepsilon \to 0,$$ 
	  in the Schwartz topology. Moreover, we shall verify that:
	  $$u_{\varepsilon} \in C^{\infty}(\R), \quad \forall \varepsilon > 0$$
	  In fact, we can deduce even more: Namely, we know that the convolution $u_{\varepsilon} = \varphi_{\varepsilon} \ast u$ grows at most polynomially, therefore it is a regular tempered distribution.\footnote{By regular we mean that the distribution is represented by a measureable function integrated against the inserted Schwartz function. We recall that for such distributions, the pointwise fractional gradient and the distributional fractional gradient agree under weak conditions, for example if $u$ is smooth and of polynomial growth as here, as seen in \eqref{auxilliary}.} Indeed, we have the following:
	  
	  \begin{lem}
	  	Let $u \in \mathcal{S}'(\R)$ be a tempered distribution and $\varphi \in \mathcal{S}(\R)$ be a Schwartz function. Then $\varphi \ast u$ is smooth and grows at most polynomially.
	  \end{lem}
	  
	  \begin{proof}
	  	To deduce smoothness, it suffices to show that $\varphi \ast u \in W^{k, \infty}_{loc}(\R)$ for any $k \in \mathbb{N}$. This can be easily done by checking that the distributional derivatives define locally bounded functions by the norm characterisation of $L^{\infty}$ using $L^{1}$-functions. To be precise, let us observe that for compactly supported, smooth $\psi$:
		$$\langle \varphi \ast u, \psi \rangle = \langle u, \varphi^{\#} \ast \psi \rangle$$
		The right hand side can be bounded by Schwartz seminorms of $\varphi^{\#} \ast \psi$. So it suffices to show that the Schwartz seminorms can be bounded by a constant depending on $\varphi$ multiplied by $\| \psi \|_{L^1}$ for all $\psi$ supported in a compact subset of $\R$, as this would imply that $\varphi \ast u \in L^{\infty}_{loc}$. The same argument applies to derivatives of the convolution, since the differentials may be absorbed in $\varphi$. More precisely, let us notice that for $\psi$ with support in $B_{r}(0)$:
		\begin{align}
			\forall N \in \mathbb{N}: |x|^{N} \Big{|} \int_{\R} \varphi(y-x) \psi(y) dy \Big{|}	&\leq \Big{|} \int_{B_r(0)} \varphi(y-x) \psi(y) (|y| + |y-x|)^{N} dy \Big{|} \notag \\
													&\lesssim \int_{B_{r}(0)} |(r + |y-x|)^N \varphi(y-x) | | \psi(y) | dy \notag \\
													&\lesssim \| \psi \|_{L^1},
		\end{align}
		where for the last inequality we used that $\varphi$ is Schwartz. Therefore, the corresponding smoothness is clear by estimation as outlined before.\\
		
		Moreover, we can see:
		$$| \varphi \ast u(x) | = \big{|} \langle u, \varphi(x - \cdot) \rangle \big{|} \lesssim p(\varphi(x- \cdot)),$$
		where $p$ is an appropriate Schwartz seminorm. Now notice:
		$$\forall N \in \mathbb{N}: (1 + | z |)^N | \partial_{x}^{m} \varphi(x - z) | \leq ( 1 + | x-z|)^N | \partial_{x}^{n} \varphi(x-z) | (1+|x|)^N,$$
		and applying these considerations to $p(\varphi(x- \cdot))$ yields the desired polynomial growth.
	  \end{proof}
	  We would like to emphasise that the very same proof works for arbitrary dimensions, i.e. for Schwartz functions and tempered distributions on $\R^n$. We shall formulate a corresponding result at the end of the current subsection.\\
	  
	  We shall make use of the fact that $u_{\varepsilon}$ represents a regular distribution. Namely, this leads to:
	  \begin{align}
	  \label{reductionueps}
	  	\langle d_{1/2} u_{\varepsilon}, G \rangle	&= \langle u_{\varepsilon}, \div_{1/2} G \rangle \notag \\
										&= \int_{\R} u_{\varepsilon}(x) \int_{\R} \frac{G(x,y) - G(y,x)}{| x-y |^{1/2}} \frac{dy}{| x-y |} dx \notag \\
										&= \int_{\R} \int_{\R} u_{\varepsilon}(x) \frac{G(x,y) - G(y,x)}{| x-y |^{1/2}} \frac{dy dx}{| x-y |} \notag \\
										&= \int_{\R} \int_{\R} u_{\varepsilon}(x) \frac{G(x,y)}{| x-y |^{1/2}} \frac{dy dx}{| x-y |} - \int_{\R} \int_{\R} u_{\varepsilon}(x) \frac{G(y,x)}{| x-y |^{1/2}} \frac{dy dx}{| x-y |}\notag \\
										&= \int_{\R} \int_{\R} u_{\varepsilon}(x) \frac{G(x,y)}{| x-y |^{1/2}} \frac{dy dx}{| x-y |} - \int_{\R} \int_{\R} u_{\varepsilon}(y) \frac{G(x,y)}{| x-y |^{1/2}} \frac{dx dy}{| x-y |}\notag \\
										&= \int_{\R} \int_{\R} \frac{u_{\varepsilon}(x) - u_{\varepsilon}(y)}{| x-y |^{1/2}} G(x,y) \frac{dy dx}{| x-y |} \notag \\
										&= \int_{\R} \int_{\R} d_{1/2} u_{\varepsilon}(x,y) G(x,y) \frac{dy dx}{| x-y |},
	  \end{align}
	  where in the last line, we used the pointwise definition of the fractional gradient given in Mazowiecka-Schikorra \cite[p.6]{schikorra}. Observe that all steps are justified by the rapid decay of the function $G$ as well as the vanishing of the derivatives along the diagonal.\\
	  
	  The main observation is now that \eqref{reductionueps} yields:
	  \begin{equation}
	  	\forall G \in \mathcal{S}_{od}(\R^2): \int_{\R} \int_{\R} F_{\varepsilon}(x,y) G(x,y) \frac{dy dx}{| x-y |} = \int_{\R} \int_{\R} d_{1/2} u_{\varepsilon}(x,y) G(x,y) \frac{dy dx}{| x-y |},
	  \end{equation}
	  which shows (by using compactly supported $G$ outside of the diagonal):
	  $$F_{\varepsilon} = d_{1/2} u_{\varepsilon}, \quad \forall \varepsilon > 0$$
	  This equality is now to be understood in a pointwise sense (almost everywhere) and no longer merely in the sense of distributions. As a result, we obtain for the pointwise fractional gradient of $u_{\varepsilon}$:
	  \begin{equation}
	  	\| F_{\varepsilon} \|_{L^{1}_{od}} = \| d_{1/2} u_{\varepsilon} \|_{L^{1}_{od}} = \int_{\R} \int_{\R} \frac{| u_{\varepsilon}(x) - u_{\varepsilon}(y) |}{ |x-y|^{\frac{3}{2}}} dx dy
	  \end{equation}
	  The expression on the right hand side of the equality corresponds to the Gagliardo-Sobolev inequality appearing, for example, in Bourgain-Brezis-Mironescu \cite[Theorem 1]{bbm}, if $d=1, s= 1/2, p=1$. Therefore, by the local boundedness of $u_{\varepsilon}$, the inequality in Bourgain-Brezis-Mironescu \cite[Theorem 1]{bbm} becomes applicable due to an approximation argument, leading to:
	  \begin{equation}
	  	\int_{-m}^{m} | u_{\varepsilon} - \overline{u_{\varepsilon,m}} |^2 dx \leq C \| d_{1/2} u_{\varepsilon} \|_{L^{1}_{od}}^2 \leq C \| d_{1/2} u \|_{L^{1}_{od}}^2 < \infty,
	  \end{equation}
	  where we denote by $\overline{u_{\varepsilon,m}}$ the average of $u_{\varepsilon}$ over the interval $[-m, m]$. Here, the constant $C > 0$ is independent of $m$ and $u$ by Remark 1 in \cite{bbm} using scale-invariance. For any $\varepsilon > 0$ and $m \in \mathbb{N}$, we define now:
	  \begin{equation}
	  \label{defveps}
	  	v_{\varepsilon,m} := u_{\varepsilon} - \overline{u_{\varepsilon, m}}
	  \end{equation}
	  For a fixed $\varepsilon > 0$, we may now deduce that this sequence converges weakly in $L^{2}_{loc}(\R)$ to some $w_{\varepsilon} \in L^{2}_{loc}(\R)$ due to the boundedness of $L^2$-norms over the compact intervals $[-m,m]$. Moreover, we can easily deduce by the weak convergence in $L^2([-m, m])$ and the lower semi-continuity of the norm:
	  $$\| w_{\varepsilon} \|_{L^2([-m,m])} \leq \liminf_{n \to \infty} \| v_{\varepsilon,m} \|_{L^2([-m,m])} \leq C \| d_{1/2} u \|_{L^{1}_{od}}$$
	  As the constant $C > 0$ is independent of $m$, we find by monotone convergence:
	  $$\| w_{\varepsilon} \|_{L^2} \leq C \| d_{1/2} u \|_{L^{1}_{od}}$$
	  Therefore, we now know $w_{\varepsilon} \in L^2(\R)$. Moreover, letting $\psi \in C^{\infty}_{c}(\R)$ and taking $N$ large enough, such that $\operatorname{supp} \psi \subset [-N/2,N/2]$, we see for any $m \geq N$:
	  $$\langle v_{\varepsilon,m}, \psi \rangle = \langle u_{\varepsilon}, \psi \rangle - \overline{u_{\varepsilon,m}} \int_{\R} \psi(x) dx$$
	  Since $v_{\varepsilon,m}$ converges weakly to $w_{\varepsilon}$, the right hand side of the equation converges to $\langle w_{\varepsilon}, \psi \rangle$. Since the first summand on the right hand side is independent of $m$, therefore the sequence $( \overline{u_{\varepsilon,m}} )_{m \in \mathbb{N}}$ must converge to some value $\tilde{C}_{\varepsilon} \in \R$. This immediately implies:
	  $$w_{\varepsilon} = u_{\varepsilon} - \tilde{C}_{\varepsilon} \in L^{2}(\R)$$
	  Repeating the same argument for $u_{\varepsilon} - \tilde{C}_{\varepsilon}$ and using distributional convergence $u_{\varepsilon} \to u$ shows that there is a constant $\tilde{C} \in \R$, such that $u - \tilde{C} \in L^{2}(\R)$ as well as:
	  \begin{equation}
	  	\| u - \tilde{C} \|_{L^2} \leq C \| d_{1/2} u \|_{L^{1}_{od}},
	  \end{equation}
	  which is the desired result. \qed \\
	  
	  One observes that we have the following immediate corollary analogous to the classical statement for the classical gradient:
	  \begin{cor}
	  \label{cor1}
	  	If $u \in \mathcal{S}^\prime (\R)$ has vanishing fractional $1/2$-gradient, then $u$ is constant.
	  \end{cor}
	  Moreover, all of the considerations in this subsection easily generalize to arbitrary dimension and thus $u \in \mathcal{S}'(\R^n)$. Therefore, we can conclude that the same result holds, if we replace $\R$ by $\R^n$ everywhere and change the weights in our integrals appropriately, invoking the results in our introductory section. Naturally, the application of Bourgain-Brezis-Mironescu \cite[Theorem 1]{bbm} necessitates differing H\"older exponents for the $L^{p}$-spaces than in the case treated above, but the argument remains valid as is, if we adapt them to the fractional Sobolev inequalities. For the reader's convenience, we state the result here:
	  
	  \begin{thm}
	  \label{generall1est}
	  	Let $p \geq 1, q \in ]1, +\infty[, s \in (0,1)$ with $sp < n$ be such that:
		$$\frac{1}{q} = \frac{1}{p} - \frac{s}{n}$$
		Moreover, let $u \in \mathcal{S}^{\prime}(\R^{n})$ be such that:
		$$d_{s} u \in L^{p}_{od}(\R^{2n})$$
		Then we have that $u - C_{u} \in L^{q}(\R^n)$ for some $C_{u} \in \R$ together with an estimate:
		\begin{equation}
		\label{generall1.2}
			\| u - C_{u} \|_{L^{q}} \lesssim \| d_{s} u \|_{L^{p}_{od}}
		\end{equation}
	  \end{thm}
	  
	  \begin{proof}
	  	Firstly, as before, using a mollifying kernel $\varphi \in \mathcal{S}(\R^n)$, we can assume wlog that $u$ is a smooth function of polynomial growth. Using \eqref{generall1young}, we can see that the $L^{p}_{od}(\R^{2n})$-bounds can be chosen uniformly as before, therefore the result could then be obtained by weak convergence due to the reflexivity of $L^{q}(\R)$. Consequently, we can deduce the following by the same sequence of inequalities as in \eqref{reductionueps} up to obvious modifications:
		$$\iint \frac{| u(x) - u(y) |^{p}}{| x - y |^{n + sp}} dx dy = \| d_{s} u \|_{L^{p}_{od}}^{p} < + \infty$$
		Let us recall the following result in Nezza-Palatucci-Valdinocci \cite[Theorem 6.7]{hitchhiker}:
		
		\begin{thm}
		\label{poincarefrac}
			Let $s \in (0,1)$ and $p \in [1, + \infty[$ such that $sp < n$ and define $q$ as above. Then there is a positive constant $C = C(n,p,s, Q) > 0$, $Q$ being a cube in $\R^n$, such that for any $f \in W^{s,p}(Q)$, we have:
			$$\| f \|_{L^{q}(Q)} \leq C \Big{(} \| f \|_{L^{p}(Q)} + \big{(} \iint_{Q \times Q} \frac{| f(x) - f(y) |^{p}}{| x-y |^{n+sp}} dx dy \big{)}^{1/p} \Big{)}$$
			Moreover, there exists a $\tilde{C} = \tilde{C}(n,p,s) > 0$ independent of the cube $Q$, such that:
			$$\big{\|} f - \dashint_{Q} f dx \big{\|}_{L^{q}(Q)} \leq \tilde{C} \Big{(} \iint_{Q \times Q} \frac{| f(x) - f(y) |^{p}}{| x-y |^{n+sp}} dx dy \Big{)}^{1/p}$$
		\end{thm}
		
		The first half of the statement is contained in Theorem 6.7 in Nezza-Palatucci-Valdinocci \cite{hitchhiker}, as cubes are Lipschitz domains. The second part can be proven by contradiction as for the classical Poincar\'e inequality: Observe that by scaling and translations, it suffices to consider the unit cube at the origin and then assume that $(f_k)$ is a sequence with $\dashint_{Q} f_k dx = 0$, such that $\| f_k \|_{W^{s,p}(Q)} = 1$, thus by the first half of the result:
		$$\| f_k \|_{L^{q}(Q)} \leq C$$
		Moreover, we may assume that:
		$$\big{\|} f_k \big{\|}_{L^{q}(Q)} \geq k \cdot \Big{(} \iint_{Q \times Q} \frac{| f_{k}(x) - f_{k}(y) |^{p}}{| x-y |^{n+sp}} dx dy \Big{)}^{1/p}$$
		This shows due to the $L^{q}$-bounds:
		\begin{equation}
		\label{eq001}
			[ f_k ]_{W^{s,p}(Q)} \leq \frac{C}{k} \to 0, \quad \text{ as } k \to \infty
		\end{equation}
		Extracting a weakly convergent subsequence, we may assume that $f_k \to f$ weakly in $L^{q}$. By using $G \in \mathcal{S}_{od}(\R^{2n})$, it follows:
		$$[ f ]_{W^{s,p}(Q)} = 0,$$
		thus $f$ is constant. This is a simple consequence from observing that $(f_k (x) - f_k (y))/| x-y |^s$ is bounded in $L^{p}_{od}(Q \times Q)$ (the notation is to be understood in the natural way) and even converges to $0$ in norm. Therefore, $d_s f_k \to 0$ in norm, but by the duality established, this means by evaluating on $G \in \mathcal{S}_{od}(\R^{2n})$ with support in $Q \times Q$:
		$$d_s f = 0 \quad \text{ on } Q \times Q$$ 
		The conclusion now follows by noticing that $d_s f$ can be understood in the pointwise sense due to convergence of the integrals involved in the Gagliardo-Sobolev seminorm for $f$ on appropriate subsets of $Q \times Q$ avoiding the diagonal and using an approximation by compactly supported off-diagonal Schwartz functions.\\
		%Let us highlight that we trivially extended all functions in the argument to the entire space and chose $G$ to be supported in $Q \times Q$.\\
		
		However, as we may also assume by Theorem 7.1 in Nezza-Palatucci-Valdinocci \cite{hitchhiker} that $f_k$ converges to $f$ strongly in $L^p$, we also have:
		$$\dashint_{Q} f dx = 0$$
		This implies that $f = 0$. This yields the desired contradiction, as $\| f_k \|_{L^{p}(Q)}$ is bounded from below due to us being able to assume a priori that $\| f_{k} \|_{W^{s,p}(Q)} = 1$ for all $k \in \mathbb{N}$. Thus, we have proven the fractional Poincar\'e inequality.\\
		
		The proof can now be completed as before applying Theorem \ref{poincarefrac} on cubes and exploiting $L^{q}_{loc}(\R^n)$-weakly convergent subsequences.
	  \end{proof}
	  
	  Naturally, Corollary \ref{cor1} has the following generalisation due to $1 \cdot s = s < n$ for all $s \in (0,1)$ and $n \in \mathbb{N}$:
	  
	  \begin{cor}
	  \label{cor2}
	  	If $u \in \mathcal{S}^\prime (\R^n)$ has vanishing fractional $s$-gradient for some $s \in (0,1)$, then $u$ is constant.
	  \end{cor}

	  \subsection{The $H^{-1/2}$-estimate and a Survey of Definitions}
	  
	  The goal of this subsection is to study some spaces appropriate for Bourgain-Brezis type estimates in the spirit of \cite{wett}. The most obvious space is, however, also the least interesting one. So while we immediately obtain the desired result, there are no new ideas involved in the proof. More intriguing choices are then considered and we attempt to see whether they are suitable or not for our purposes.
	  
	  \subsubsection{A first Definition of $H^{-1/2}_{od}(\R^2)$}
	  First, we would like to obtain a result of the following form:
	  $$d_{1/2} u \in H^{-1/2}_{od}(\R^2)\ \Rightarrow\ u \in L^{2}(\R)$$
	  This is in fact the kind of behaviour we would expect from a suitable generalisation of a negative exponent Sobolev space to our new distributional setup. In order to achieve this, we first have to make sense of the space $H^{-1/2}_{od}(\R^2)$. A natural definition would be:
	  \begin{defi}
	  \label{d1/2image}
	  	We define:
		\begin{equation}
		\label{defh1/2}
			H^{-1/2}_{od}(\R^2) := \big{\{} U \in \mathcal{S}^{\prime}_{od}(\R^2)\ \big{|}\ U = d_{1/2} u \text{ for some } u \in L^{2}(\R) \subset \mathcal{S}^{\prime}(\R)  \big{\}} = \text{Im } d_{1/2}
		\end{equation}
		Moreover, the norm is defined by:
		$$\| U \|_{H^{-1/2}_{od}} = \| d_{1/2} u \|_{H^{-1/2}_{od}} := \| u \|_{L^2}$$
	  \end{defi}
	  The definition is inspired by the dual of the homogeneous Sobolev spaces which can be characterised as sums of distributional partial derivatives of functions of low regularity. A characterisation of this space based merely on $\mathcal{S}^{\prime}_{od}(\R^2)$ might exist. We would like to emphasise at this point already, that this definition of $H^{-1/2}_{od}(R^2)$ essentially will render the generalised Bourgain-Brezis inequality a trivial corollary to our consideration in the previous subsection for fractional gradients in $L^{1}_{od}$. This is in stark contrast to the situation encounter in the local case and is likely due to the assumptions on $H^{-1/2}_{od}(\R^2)$ being too strong. Therefore, we shall explore some relaxations later on in the current subsection.\\
	  
	  The following statement is now obvious by what we have done before:
	  
	  \begin{thm}
	  \label{h1/2est}
	  	Assume that $u \in \mathcal{S}^\prime(\R)$ is such that:
		$$d_{1/2} u = U \in H^{-1/2}_{od}(\R^2) \subset \mathcal{S}^{\prime}_{od}(\R^2),$$
		then there exists a constant $C \in \R$ such that $u - C \in L^{2}(\R)$ together with the estimate:
		\begin{equation}
		\label{mainest}
			\| u - C \|_{L^2} = \| d_{1/2} u \|_{H^{-1/2}_{od}}
		\end{equation}
	  \end{thm}
	  The proof merely consists of noticing that by definition, there exists a $v \in L^{2}(\R)$ such that:
	  $$d_{1/2} v = d_{1/2} u$$
	  Therefore:
	  $$d_{1/2} (u-v) = 0,$$
	  and Corollary \ref{cor1} immediately shows that $u-v$ is a constant. Thus, we can deduce:
	  $$\exists C \in \R: u = v+C,$$
	  and consequently:
	  $$\| u-C \|_{L^2} = \| v \|_{L^2} = \| U \|_{H^{-1/2}_{od}},$$
	  which is the desired result. \qed \\
	  
	  Once again, there is no need to restrict to the $1$-dimensional case, provided we define in an analogous manner Sobolev spaces with negative exponents for arbitrary H\"older exponents as above. This generalisation shall be treated later on.\\
	  
	  To render the previous result slightly more interesting, let us provide another characterisation of the space $H^{-1/2}_{od}(\R^2)$ concealing the immediate connection to $L^{2}$-fractional gradients and replacing it by a continuity property that is more in line with the type of characterisation one might expect:
	  
	  \begin{lem}
	  \label{altcharacterisationh1/2}
	  	Let $U \in \mathcal{S}^{\prime}_{od}(\R^2)$ be any off-diagonal distribution. Then $U \in H^{-1/2}_{od}(\R^2)$ if and only if:
		\begin{equation}
		\label{altcharah12}
			| \langle U, G \rangle | \leq C \| \div_{1/2} G \|_{L^{2}}, \quad \forall G \in \mathcal{S}_{od}(\R^2),
		\end{equation}
		for some $C > 0$, a suitable constant independent of $G$. Moreover, the smallest such constant $C$ actually agrees with $\| U \|_{H^{-1/2}_{od}}$.
	  \end{lem}
	  
	  \begin{proof}
	  	Let us observe that if $U \in H^{-1/2}_{od}(\R^2)$, then by definition:
		$$d_{1/2} u = U,$$
		for some $L^{2}$-function $u$ on $\R$. Therefore, by duality:
		\begin{align}
			| \langle U, G \rangle |	&= | \langle d_{1/2} u, G \rangle | \notag \\
								&= | \langle u, \div_{1/2} G \rangle | \notag \\
								&\leq \| u \|_{L^2} \| \div_{1/2} G \|_{L^2},
		\end{align}
		where we used H\"older's inequality in the last step. Therefore, the estimate in \eqref{altcharah12} holds with $C = \| u \|_{L^2}$.\\
		
		Let us now assume the converse, namely:
		$$\forall G \in \mathcal{S}_{od}(\R^2): | \langle U, G \rangle | \leq C \| \div_{1/2} G \|_{L^2}$$
		We observe that this implies that the following operator:
		$$L_{U}: \big{\{} \div_{1/2} G\ |\ G \in \mathcal{S}_{od}(\R^2) \big{\}} \to \R, \quad \div_{1/2} G \mapsto \langle U, G \rangle,$$
		is well-defined, linear and continuous with respect to the $L^{2}$-norm. Indeed, if $\div_{1/2} G = \div_{1/2} \tilde{G}$ for $G, \tilde{G} \in \mathcal{S}_{od}(\R^2)$, then we know:
		$$| \langle U, G \rangle - \langle U, \tilde{G} \rangle | = | \langle U, G - \tilde{G} \rangle | \leq C \| \div_{1/2} (G - \tilde{G} ) \|_{L^2} = 0$$
		Furthermore, linearity is obvious and the continuity with respect to the $L^2$-norm is exactly the assumption on $U$ above. Therefore, by Hahn-Banach's extension result and Riesz' representation theorem, this means that we can extend $L_{U}$ to a continuous linear functional on all of $L^{2}(\R)$ and consequently that there is a $L^2$-function $u$, such that:
		$$\langle u, \div_{1/2} G \rangle = \langle U, G \rangle, \quad \forall G \in \mathcal{S}_{od}(\R^2),$$
		and as a result, using the dual definition of the fractional gradient:
		$$d_{1/2} u = U,$$
		where equality holds in the sense of off-diagonal distributions. Observe that the norm of $U$ in $H^{-1/2}_{od}(\R^2)$ can also be characterised either as the $L^2$-norm of $u \in L^{2}(\R)$ with $d_{1/2} u$ or as the norm of the induced linear functional $L_{U}$ above. This follows, as the Hahn-Banach extension can be chosen to satisfy the same norm properties as $L_{U}$ and thus $\| u \|_{L^{2}}$ actually coincides with the operator norm.
	  \end{proof}
	  
	  This computation naturally suggests the following definition by duality: The space $\dot{H}^{1/2}_{od}(\R^2)$ is the completion of the space $\mathcal{S}_{od}(\R^2)$ under the semi-norm:
	  $$\| G \|_{\dot{H}^{1/2}_{od}} := \| \div_{1/2} G \|_{L^2} = \Big{(} \int_{\R} \big{|} \int_{\R} \frac{2 G(x,y)}{| x-y |^{3/2}} dy \big{|}^{2} dx\Big{)}^{1/2}$$
	  To further examine this space, a nice characterisation of the $G$ with vanishing $1/2$-divergence is necessary. One should note that the definition of $\dot{H}^{1/2}_{od}(\R^2)$ and $H^{-1/2}_{od}(\R^2)$ naturally generalize for $s \in (0,1)$ as follows:
	  \begin{equation}
	  	W^{-s,p}_{od}(\R^{2n}) := \big{\{} U \in \mathcal{S}^{\prime}_{od}(\R^{2n})\ \big{|}\ U = d_{s} u \text{ for some } u \in L^{p}(\R^n) \big{\}}
	  \end{equation}
	  together with the norm:
	  \begin{equation}
	  	\forall U \in W^{-s,p}_{od}(\R^{2n}): \| U \|_{W^{-s,p}_{od}} := \| u \|_{L^{p}},
	  \end{equation}
	  for $u \in L^{p}(\R^n)$ as above. The norm is well-defined due to Corollary \ref{cor2}. Similarily, we can define $\dot{W}^{s,p}_{od}(\R^{2n})$ in analogy to $\dot{H}^{1/2}_{od}(\R^2)$. We observe that an alternative characterisation in the spirit of Lemma \ref{altcharacterisationh1/2} can be obtained using duality of $L^p$-spaces.\\
	  
	  The following result can be obtained completely analogous to Theorem \ref{h1/2est}:
	  
	  \begin{thm}
	  \label{wspest}
	  	Assume that $u \in \mathcal{S}^\prime(\R^n)$ is such that:
		$$d_{s} u = U \in W^{-s,p}_{od}(\R^2) \subset \mathcal{S}^{\prime}_{od}(\R^{2n}),$$
		for some $s \in (0,1)$ and $p \in [1,\infty[$, then there exists a constant $C \in \R$ such that $u - C \in L^{p}(\R^n)$ together with the estimate:
		\begin{equation}
		\label{mainest.2}
			\| u - C \|_{L^p} = \| d_{s} u \|_{W^{-s,p}_{od}}
		\end{equation}
	  \end{thm}
	  
	  The proof is exactly the same as for Theorem \ref{h1/2est}, except for the use of Corollary \ref{cor2} instead of Corollary \ref{cor1}.\\

	  \subsubsection{A (not so successful) Variation of the Definition}
	  
	  Another interesting observation (based on the definition of fractional Sobolev spaces in Gaia \cite[p.19-20]{gaia}) is the following: We could also directly define for $G \in \mathcal{S}_{od}(\R^2)$ that $G \in \dot{\tilde{H}}^{1/2}_{od}(\R^2) =: X$, if the following norm is finite:
	  $$\| G \|_{X} := \Big{\|} \frac{| G(x,y) |}{|x-y|^{1/2}} \Big{\|}_{L^{2}_{od}}$$
	  To ensure suitable properties of the entire space $X$, we use completion with respect to the norm. Then, we could also consider the dual space $X^{\ast}$ of this space $X$ instead of the previously defined $H^{-1/2}_{od}(\R^2)$. It can be seen for $U \in X^{\ast}$ (notice that multiplication by functions of at most polynomial growth is justified as for Schwartz functions):
	  $$\langle U, G \rangle = \langle | x-y |^{1/2} U, \frac{G}{| x-y |^{1/2}} \rangle,$$
	  which implies:
	  $$| x-y |^{1/2} U \in L^{2}_{od}(\R^2)$$
	  The converse holds trivially, so this is an alternative characterisation of $X^{\ast}$. Therefore, if $d_{1/2} u \in X^{\ast}$, then:
	  $$| x-y |^{1/2} d_{1/2} u \in L^{2}_{od}(\R^2)$$
	  A direct examination shows:
	  \begin{align}
	  \label{todnull}
	  	\langle | x-y |^{1/2} d_{1/2} u, G \rangle	&= \langle d_{1/2} u, | x-y |^{1/2} G \rangle \notag \\
										&= \langle u, \div_{1/2} \big{(} | x-y |^{1/2} G \big{)} \rangle \notag \\
										&= \langle u, \div_{0} G \rangle \notag \\
										&= \langle d_{0} u, G \rangle,
	  \end{align}
	  where we extend the definition of $d_{1/2}$ and $\div_{1/2}$ as expected to $s = 0$. Observe that all previously established properties continue to hold also for $s=0$, as we have already seen above. We may conclude:
	  $$d_{0} u \in L^{2}_{od}(\R^2)$$
	  Let us present some further ideas and observations related to this definition: We observe that by \eqref{todnull}, we may deduce:
	  	$$d_{0} u = H \in L^{2}_{od}(\R^2)$$
	  	Let us now assume that $u$ is smooth and of polynomial growth. The general case follows by mollification, as mollifying with a Schwartz kernel as in the case of an $L^1$-estimate leads to this scenario. Then, we have for compactly supported, smooth, off-diagonal $G$:
	  	\begin{align}
	  		\langle d_0 u, G \rangle	&= \langle u, \div_{0} G \rangle \notag \\
								&= \big{\langle} u, \int \frac{2 G(x,y)}{|x-y|} dy \big{\rangle} \notag \\
								&= \iint \frac{2 u(x) G(x,y)}{|x-y|} dx dy \notag \\
								&= \iint \frac{(u(x) - u(y)) G(x,y)}{| x-y |} dx dy, 
		\end{align}
	  	and therefore, by comparison and due to the compactly supported functions $G$ being dense in $L^2_{od}(\R^2)$:
	  	$${u(x) - u(y)} = d_0 u = H(x,y) \in L^2_{od}(\R^2)$$
	  	But this implies that:
	  	\begin{equation}
			\iint \frac{| u(x) - u(y) |^2}{| x-y |} dx dy < + \infty
		\end{equation}
		This train of thought reveals a flaw in the definition: The finiteness of the above integral is a unnecessarily strong assumption, as not all $f \in L^2(\R)$ do satisfy the boundedness of the integral. As an example, let us show that for $\chi_{]-1,1[}$, the corresponding integral is not finite, thus showing that $u$ could not be the characteristic function of $]-1,1[$ and hence restricting the set of possible $L^2$-functions, which should intuitively not be the case for the kind of space we are looking for:
		\begin{align*}
			\iint \frac{| u(x) - u(y) |^2}{| x-y |} dx dy	&= 2 \int_{\R \setminus ]-1,1[} \int_{-1}^{1} \frac{1}{| y - x |} dxdy \\
											&\geq 2 \int_{1}^{+\infty} \int_{-1}^{1} \frac{1}{y-x} dx dy \\
											&= 2 \int_{1}^{+\infty} \log \big{(} \frac{y + 1}{y - 1} \big{)} dy \\
											&= 4 \int_{1}^{+\infty} \frac{\log(z)}{(z-1)^2} dz = + \infty
		\end{align*}
		Observe that at one step we restricted to only one half of the integral for convenience as well as a change of variables $z = (y+1)/(y-1)$. We used that:
		$$\frac{dz}{dy} = - \frac{2}{(y-1)^2} = - \frac{2}{\big{(} \frac{z+1}{z-1} - 1 \big{)}^2} = - \frac{1}{2} (z-1)^2$$
		
		Thus, this kind of definition does not seem to be appropriate for our purposes.

	  \subsection{The Combined Estimate in Arbitrary Dimensions}
	  
	  The following result is an immediate consequence of the definition of the space $H^{-1/2}_{od}(\R^2)$ combined with the estimates proven in Theorem \ref{l1est} and \ref{h1/2est}:
	  
	  \begin{thm}
	  \label{combest}
	  	Let $u \in \mathcal{S}^{\prime}(\R)$ be given and assume that:
		$$d_{1/2} u \in L^{1}_{od} + H^{-1/2}_{od}(\R^2)$$
		Then $u - C_{u} \in L^2(\R)$ for some constant $C_{u} \in \R$ depending on $u$ together with the estimate:
		\begin{equation}
		\label{estthm3}
			\| u - C_{u} \|_{L^2} \leq C \| d_{1/2} u \|_{L^{1}_{od} + H^{-1/2}_{od}},
		\end{equation}
		where $C > 0$ is independent of $u$.
	  \end{thm}
	  
	  The proof consists of decomposing $d_{1/2} u$ into a $L^1$- and a $H^{-1/2}_{od}$-part. Then find $v \in L^2(\R)$, such that the $H^{-1/2}_{od}$-part of $d_{1/2} u$ is precisely $d_{1/2} v$. Existence is ensured by the very definition of the space. Consequently, we have reduced the problem to the $L^1$-estimate and we are done. The estimate in \eqref{estthm3} follows by optimisation over all decompositions of $d_{1/2} u$.
	  
	  \begin{proof}
	  	Assume that $u$ is a tempered distribution as in the Theorem \ref{combest}. Then we decompose:
		$$d_{1/2} u = U + V,$$
		where $U \in L^1_{od}(\R^2)$ and $V \in H^{-1/2}_{od}(\R^2)$. We recall that by Definition \ref{d1/2image}, then there exists a $v \in L^2(\R)$, such that:
		$$d_{1/2} v = V$$
		Thus, as $v \in L^2(\R) \subset \mathcal{S}^{\prime}(\R)$, we obviously have for $\tilde{u} := u - v$:
		\begin{equation}
			d_{1/2} \tilde{u} = U \in L^{1}_{od}(\R^2)
		\end{equation}
		As a result, by the estimate in Theorem \ref{l1est}, there exists a constant $C_{\tilde{u}} \in \R$, such that $\tilde{u} - C_{\tilde{u}} \in L^2(\R)$ and:
		$$\| \tilde{u} - C_{\tilde{u}} \|_{L^2} \lesssim \| d_{1/2} \tilde{u} \|_{L^{1}_{od}} = \| U \|_{L^1_{od}}$$
		By the choice of $v$, we also have by Theorem \ref{h1/2est} and Definition \ref{d1/2image}:
		$$\| v \|_{L^2} = \| V \|_{H^{-1/2}_{od}},$$
		which immediately implies by the triangle inequality that $u - C_{\tilde{u}} \in L^2(\R)$ and more precisely:
		\begin{equation}
			\| u - C_{\tilde{u}} \|_{L^{2}} \leq \| \tilde{u} - C_{\tilde{u}} \|_{L^2} + \| v \|_{L^{2}} \lesssim \| U \|_{L^1_{od}} + \| V \|_{H^{-1/2}_{od}}
		\end{equation}
		Since $U, V$ were arbitrary, we can choose them in such a way that:
		$$\| U \|_{L^1_{od}} + \| V \|_{H^{-1/2}_{od}} \leq 2 \| d_{1/2} u \|_{L^{1}_{od} + H^{-1/2}_{od}},$$
		or arbitrarily small in case $\| d_{1/2} u \|_{L^{1}_{od} + H^{-1/2}_{od}} = 0$, and therefore:
		\begin{equation}
			\| u - C_{\tilde{u}} \|_{L^{2}} \lesssim \| d_{1/2} u \|_{L^{1}_{od} + H^{-1/2}_{od}}
		\end{equation}
		This establishes the desired result.
	  \end{proof}

	Precisely the same proof establishes the following, if we just apply Theorem \ref{generall1est} instead of Theorem \ref{l1est}:
	
	 \begin{thm}
	  \label{generalcombest}
	  	Let $u \in \mathcal{S}^{\prime}(\R^n)$, $s \in (0,1)$ and assume that:
		$$d_{s} u \in L^{1}_{od} + W^{-s, \frac{n}{n-s}}_{od}(\R^{2n})$$
		Then $u - C_{u} \in L^{\frac{n}{n-s}}(\R^n)$ for some constant $C_{u} \in \R$ depending on $u$ together with the estimate:
		\begin{equation}
		\label{estthm3.2}
			\| u - C_{u} \|_{L^{\frac{n}{n-s}}} \leq C \| d_{s} u \|_{L^{1}_{od} + W^{-s, \frac{n}{n-s}}_{od}},
		\end{equation}
		where $C > 0$ is independent of $u$.
	  \end{thm}
	 This combined estimate is somewhat unsatisfying, as the decomposition yields little new or interesting insight due to the nature of the spaces $W^{-s,p}_{od}(\R^{2n})$. It would therefore be interesting to investigate alternative spaces and definitions for $W^{-s,p}_{od}(\R^{2n})$ and see, how they lend themselves to comparable results.

	  \section{Fractional Sobolev embedding for $sp > n$}
	  \label{sobolevemb}
	  
	  For completeness' sake, let us briefly address the case $sp > n$. This case implies H\"older regularity by means of the fractional Sobolev inequality. The following result is an immediate consequence of Nezza-Palatucci-Valdinocci \cite[Theorem 8.1]{hitchhiker} as well as the scaling-invariance of the estimate, once another Poincar\'e-type inequality has been established:
	  
	  \begin{thm}
	  	Let $s \in (0,1), p \in [1, + \infty[$ such that $sp > n$ and $u \in \mathcal{S}^{\prime}(\R^n)$. If we have:
		$$d_{s} u \in L^{p}_{od}(\R^{2n}),$$
		then also $u \in C^{0,\alpha}(\R^{2n})$ for $\alpha = s - n/p$. The following estimate holds:
		$$[ u ]_{C^{0,\alpha}} \leq C \| d_s u \|_{L^{p}_{od}},$$
		for some $C = C(n,s,p) > 0$.
	  \end{thm}
	  It should be emphasised that the case $p = + \infty$ follows directly from the definition and is merely excluded, as we did not define $L^{\infty}_{od}(\R^{2n})$ explicitely. Naturally, one would define the corresponding norm by:
	  $$\| F \|_{L^{\infty}_{od}} := \sup_{(x,y) \in \R^{2n}} {| F(x,y) |},$$
	  where the supremum here is to be taken as the essential supremum (as per usual).
	  Then, using the pointwise estimates applied to $d_{s} u$ after regularising appropriately gives the corresponding H\"older regularity expected in this case.
	  
	  \begin{proof}
	  	As in the proof of the fractional Sobolev embedding in the case $sp < n$, we may assume by approximation that $u$ is smooth and of polynomial growth. Indeed, we can approximate $u$ by such a sequence $u_k$ and in a way that for all $k \in \mathbb{N}$:
		$$\| d_s u_k \|_{L^{p}_{od}} \leq 2 \| d_s u \|_{L^{p}_{od}},$$
		by mollification. Then we see, provided the result above holds for regular $u$:
		$$[ u_k ]_{C^{0,\alpha}} \lesssim 2 \| d_s u \|_{L^{p}_{od}},$$
		therefore the sequence $(u_k)$ is equicontinuous. By changing $u_k$ by a constant (depending on $k \in \mathbb{N}$), we may assume $u_k(0) = 0$, for all $k$. Therefore, by Arzela-Ascoli, there is a subsequence which converges uniformly to a continuous limit $w$ on compact sets. Therefore, the $u_k - u_k(0)$ also converge in the distributional sense to $w$. Thus, as the sequence $(u_k)$ converges distributionally to $u$ by its very definition, this shows that the $u_k(0)$ must be convergent as well and let us call the limit $\tilde{C}$. Then, we conclude by the distributional convergence:
		$$w = u - \tilde{C}$$
		We may deduce the desired estimate by using uniform convergence on compact subsets and investigating the $C^{0,\alpha}$-seminorm directly. Observe that we can omit $\tilde{C}$, as the seminorm ignores constants.\\
		
		It remains to establish the result for smooth $u$ of polynomial growth. This is a consequence of the following estimate for all such $u$:
		\begin{equation}
		\label{estholder}
			[ u ]_{C^{0,\alpha}(Q)} \leq C \cdot [ u ]_{W^{s,p}(Q)},
		\end{equation}
		for any cube $Q$ with a $C$ only depending on $n, s, p$. This follows by reduction to the unit cube due to scaling-invariance and the usual Poincar\'e argument: Let $u_k$ be such that $\| u_k \|_{W^{s,p}(Q)} = 1$ for all $k \in \mathbb{N}$ and assume:
		$$[ u_k ]_{C^{0,\alpha}(Q)} \geq k \cdot [ u_k ]_{W^{s,p}(Q)}$$
		Moreover, we may assume $\dashint_{Q} u_k dx = 0$. By Nezza-Palatucci-Valdinocci \cite[Theorem 8.2]{hitchhiker}, we know that $[ u_k ]_{C^{0,\alpha}(Q)}$ is bounded independent of $k$ and thus:
		$$[ u_k ]_{W^{s,p}(Q)} \to 0$$
		By the uniform bound on $\| u_k \|_{C^{0,\alpha}(Q)}$ following Nezza-Palatucci-Valdinocci \cite{hitchhiker}, we may extract a subsequence converging uniformly to some $u \in C^{0, \alpha}(Q)$. Therefore, the convergence also holds in the distributional sense. Additionally, the convergence also holds in $L^p$ by the compactness of the embedding by Nezza-Palatucci-Valdinocci \cite{hitchhiker} and as a result:
		$$\dashint_{Q} u dx = 0$$
		If we argue as in the fractional Sobolev embedding for $sp < n$, we may deduce that $d_s u = 0$ in $Q \times Q$, showing that $u$ is constant. By the average condition, $u = 0$, but this contradicts the $L^p$-convergence, so the desired estimate is established by contradiction.\\
		
		The proof is now straightforward, as the bounds in \eqref{estholder} on $Q$ are independent of the size of the cube, if we estimate $[ u ]_{W^{s,p}(Q)}$ by $\| d_s u \|_{L^{p}_{od}}$, therefore translating immediately to H\"older continuity of $u$, as desired.
	  \end{proof}

\end{document}